\documentclass{article}
\usepackage{amsmath,amsthm,amssymb,graphicx}

\newcommand{\ga}{\alpha}
\newcommand{\gb}{\beta}

\newcommand{\gw}{\omega}

\newcommand{\supp}{\mathrm{supp}}

\newcommand{\dom}{\mathrm{dom}}
\newcommand{\rng}{\mathrm{rng}}

\newcommand{\liff}{\leftrightarrow}

\newcommand{\eps}{\varepsilon}

\newcommand{\power}{\mathcal{P}}
\newcommand{\aut}{\mathrm{Aut}}
\newcommand{\acts}{\curvearrowright}
\newcommand{\pstab}{\mathrm{pstab}}
\newcommand{\stab}{\mathrm{stab}}
\newcommand{\worank}{\mathrm{wor}}
\newcommand{\clirank}{\mathrm{clir}}
\newcommand{\hf}{\mathrm{HF}}

\newtheorem{theorem}{Theorem}[section]

\newtheorem{claim}[theorem]{Claim}
\newtheorem{corollary}[theorem]{Corollary}

\newtheorem{proposition}[theorem]{Proposition}

\theoremstyle{definition}
\newtheorem{definition}[theorem]{Definition}
\newtheorem{example}[theorem]{Example}

\title{Nonarchimedean groups and the axiom of choice\footnote{2020 AMS subject classification 03E25, 22F05.}}

\author{
	Justin Young\\
	University of Florida\\
	jyoung3@ufl.edu
	\and
	Jind{\v r}ich Zapletal\\
	University of Florida\\
	zapletal@ufl.edu}

\begin{document}
	\maketitle
	
	\begin{abstract}
	We prove several novel connections between properties of non-archimedean groups and fragments of the axiom of choice which hold in their associated permutation model.	
	\end{abstract}

\section{Introduction}
In this paper, we prove several theorems connecting properties of non-archimedean Polish groups with fragments of Axiom of Choice which hold in their associated permutation models. Let $X$ be a countable set with some relations and operations on it, and let $\aut(X)\acts X$ be the group of all automorphisms of $X$ viewed as a closed subgroup of the full symmetric group acting on the set $X$ by application. Consider the ideal $I$ of finite subsets of $X$; then the ideal $I$ is invariant under the action and $\aut(X)\acts X, I$ is a dynamical ideal in the sense of \cite{z:dideals}. It naturally generates a Fraenkel--Mostowski style permutation model $W[[X]]$ of ZFA, the choiceless set theory with atoms, in which $X$ is identified with its set of atoms. 

Section~\ref{pliablesection} investigates the first correspondence:

\begin{definition}
	A non-archimedean Polish group $\Gamma$ is \emph{pliable} if there is an open neighborhood of the unit such that its conjugates form a topology basis around the unit, and in addition, this is true of all open subgroups of $\Gamma$.
\end{definition}

\begin{definition}
	(ZF) The \emph{Axiom of Well-ordered Choice} \cite[Form 40]{howard:ac} is the statement that every well-orderable family of nonempty sets has a choice function.
\end{definition}

\begin{theorem}
	\label{theoremA}
	Let $X$ be a countable structure. The following are equivalent:
	
	\begin{enumerate}
		\item The group $\aut(X)$ is pliable;
		\item in $W[[X]]$, the Axiom of Well-ordered Choice holds.
	\end{enumerate}
\end{theorem}

\noindent  In the context of permutation models obtained from the ideal of finite sets, the axiom of well-ordered choice is in fact equivalent to the axiom of countable choice \cite{z:dideals}. Pliable groups form a very special subclass of feebly locally compact groups of \cite[Section 2.1]{rosendal:global}. A great part of our contribution is in fact a construction of non-discrete pliable groups. They appear to be related to the so-called Knight's model \cite{knight:model}; our construction is much simpler than that of Julia Knight though. 

Section~\ref{clisection} investigates the impact of complete group metrics on permutation models:

\begin{definition}
	A topological group $\Gamma$ is \emph{cli} if it carries a complete left-invariant metric.
\end{definition}

\begin{definition}
	(ZF) The \emph{well-orderability hierarchy} is defined by recursion on ordinal $\ga$: $W_0=$ the class of well-orderable sets, $W_{\alpha+1}=$the class of all unions of well-orderable subsets of $W_\ga$, and $W_\ga=\bigcup_{\gb\in\ga}W_\gb$ if $\ga$ is limit.
\end{definition}

\begin{theorem}
	\label{theoremB}
	Let $X$ be a countable structure. The following are equivalent:
	
	\begin{enumerate}
		\item the group $\aut(X)$ is cli;
		\item in $W[[X]]$, every set appears at some stage of the well-orderability hierarchy.
	\end{enumerate}
\end{theorem}

\noindent An extreme case of the theorem occurs when the acting group is abelian. In such a case, in the permutation model every set is a union of well-orderable family of finite sets \cite[Theorem 6.19]{z:dideals}, meaning that the well-orderability hierarchy reaches every set on its second level. At the same time, every compatible left-invariant matric in an abelian group is two-sided invariant, therefore complete. In general, the well-orderability hierarchy follows the rank on cli groups introduced by \cite{panagiotopoulos:ranks}; in particular, it reaches every set at a countable stage.

Note that a pliable group cannot be cli unless it is discrete, in which case the permutation model satisfies the full Axiom of Choice. Thus, apart from this trivial case, Theorems~\ref{theoremA} and~\ref{theoremB} describe completely separate contexts.

Sections~\ref{bootstrappingsection} and ~\ref{finalsection} investigate a group-theoretic notion which is key for the uniform treatment of many permutation models.

\begin{definition}
	A countable structure $X$ admits \emph{classification of open subgroups} if for every open subgroup $\Delta\subseteq\aut(X)$ there is a finite set $a\subset X$ such that $\pstab(a)\subseteq\Delta\subseteq\stab(a)$.
\end{definition}

\noindent This is a weakening of the strong small index property of a structure, applied to open subgroups of $\aut(X)$ only, instead of all subgroups of index smaller than continuum. The strong small index property has been verified for great many structures: see \cite{truss:smallindex} for the rationals with their ordering or the countable atomless Boolean algebra, \cite{cameron:random} for the Rado graph, or \cite{Sh:1108} for many other homogeneous structures. We show that for relational Fraiss{\' e} classes with disjoint amalgamation, there is a simple characterization when classification of open subgroups holds. In addition, classification of open subgroups greatly simplifies the study of the associated permutation model.

Section~\ref{bootstrappingsection} isolates the ubiquitous class of \emph{bootstrapping Fraiss{\' e} classes} (Definition~\ref{bootstrappingdefinition}). The key characterization result is the following.

\begin{theorem}
	\label{theoremC}
	Let $\mathcal{F}$ be a relational Fraiss{\' e} class with disjoint amalgamation with limit $X$. The following are equivalent:
	
	\begin{enumerate}
		\item $\mathcal{F}$ is a bootstrapping class;
		\item $X$ has classification of open subgroups.
	\end{enumerate}
\end{theorem}

\noindent Section~\ref{finalsection} shows how to apply classification of open subgroups to the theory of the derived permutation model $W[[X]]$. The following theorem makes it possible to reduce the study of arbitrary sets in $W[[X]]$ to the study of $X$. In choiceless set theory, a set $Y$ is a \emph{support set} if for every set $A$ there is an ordinal $\ga$ and an injection from $A$ to $Y\times\ga$. For a set $Y$, write $\hf(Y)$ for the smallest set containing $X$ as a subset and closed under the formation of finite subsets.

\begin{theorem}
	\label{theoremD}
	If $X$ admits classification of open subgroups, then in $W[[X]]$, $\hf(X)$ is a support set.
\end{theorem}

\noindent This should be compared to the axiom called SVC (Small Violation of Choice) by Blass \cite{blass:svc}. Further theorems in Section~\ref{finalsection} show how the conclusion of Theorem~\ref{theoremD} can be applied abstractly. For example, in $W[[X]]$ every set is linearly orderable if and only if $X$ is, every set of finite sets has a selector if and only if the set of all finite subsets of $X$ has a selector etc. This leads to a uniform and improved treatment of many well-known permutation models of ZFA and their generalizations at the end of the paper.

The notation in this paper follows the set-theoretic standard of \cite{jech:newset}; in matters of permutation models, the notation of \cite{z:dideals}. The construction of the permutation model can be found in \cite{z:dideals}, \cite{howard:ac}, \cite{jech:newset} and many other texts. An important point is that the action of the group $\aut(X)$ on $X$ canonically extends to action on the whole permutation model $W[[X]]$ by $\in$-automorphisms. Composition of functions $f$ and $g$ is denoted by $fg$; $(fg)(x)=f(g(x))$. If a group $\Gamma$ acts on a set $X$, we write $\stab(x)=\{\gamma\in\Gamma\colon \gamma\cdot x=x\}$ and for a set $a\subset X$, $\pstab(x)=\bigcap_{x\in a}\stab(x)$. In the context of relational Fraiss{\' e} classes, a finite structure is conflated with its domain. For a set $X$, $[X]^{<\aleph_0}$ is the set of all finite subsets of $X$ and $\hf$ is the smallest set containing $X$ as a subset and closed under formation of finite subsets. 

The work on this paper was partially supported by NSF grant DMS 2348371.

\section{Pliable groups}
\label{pliablesection}

\begin{theorem}
	Let $X$ be a countable structure. The following are equivalent:
	
	\begin{enumerate}
		\item the group $\aut(X)$ is pliable;
		\item in $W[[X]]$, the Axiom of Well-ordered Choice holds.
	\end{enumerate}
\end{theorem}

\begin{proof}
This is just a diagram-chasing argument using some foundational results of \cite{z:dideals}. Let $X$ be a countable structure and $\aut(X)$ act on $X$ by application. For each finite set $a\subset X$ let $\bar a\subset X$ be the dynamical closure of $a$, the set $\{x\in X\colon \forall\gamma\in\pstab(a)\ \gamma\cdot x=x\}$. Note that $\pstab(a)=\pstab(\bar a)$ holds. Write $I$ for the ideal of finite sets on $X$, and $J$ for the ideal generated by sets $\bar a$ for $a\in I$. Then $J$ is the dynamical closure of $I$, and the permutation models derived from $\aut(X)\acts X, I$ and $\aut(X)\acts X, J$ coincide \cite[Proposition 2.13]{z:dideals}. By \cite[Theorem 3.3]{z:dideals}, the axiom of well-ordered choice is equivalent to cofinal orbits for $J$, which translates to the statement that for every $a\in I$ there is $b\in I$ containing $a$ such that for every $c\in I$ containing $a$, there is $\gamma\in\pstab(a)$ such that $c\subset\overline{\gamma b}$.

It will be enough to show that the latter statement is equivalent to the pliability of the group $\aut(X)$. Suppose first that the group is pliable and $a\subset X$ is a finite set. Then $\pstab(a)$ is an open subgroup of $\aut(X)$; by the pliability assumption, it has an open subgroup $\Delta$ such that conjugates of $\Delta$ within $\pstab(a)$ form a neighborhood basis for $1$. Let $b\supset a$ be a finite set such that $\pstab(b)\subset\Delta$; we claim that $b$ works as in the restatement of cofinal orbits in the previous paragraph. Indeed, if $c\supset a$ is a finite set, then by the choice of $\Delta$ there is an element $\gamma\in\pstab(a)$ such that $\gamma\Delta\gamma^{-1}\subseteq\pstab(c)$. This means that $\gamma\pstab(b)\gamma^{-1}\subset\pstab(c)$, which turns into $\pstab(\gamma\cdot b)\subset\pstab(c)$ and $c\subset\overline{\gamma b}$ as desired.

Now suppose that the ideal $J$ has cofinal orbits. To argue for pliability of $\aut(X)$, let $\Delta$ be an open subgroup. Let $a\subset X$ be a finite set such that $\pstab(a)\subset\Delta$. Let $b$ be a finite set witnessing cofinal orbits; we claim that within $\pstab(a)$, the conjugates of $\pstab(b)$ form a basis around the unit. To see this, let $c\subset X$ be a finite set containing $a$. By the cofinal orbits assumption, there is $\gamma\in\pstab(a)$ such that $c\subseteq\overline{\gamma b}$. Unraveling the action, the latter inclusion turns into $\gamma\pstab(b)\gamma^{-1}\subset\pstab(c)$ as desired.
\end{proof}

As a final contribution to this section, we produce an example of non-discrete pliable group. It is constructed as a closed subgroup of the automorphism group of rationals with their usual ordering. The main tool is the following finite approximation device.

\begin{definition}
	Define the poset $P$ as follows: a condition $p\in P$ is a finite subset of $\mathbb{Q}^3$ satisfying the following conditions:
	
	\begin{enumerate}
		\item if $\langle a, s, t\rangle\in p$, then either $a<s<t$ or $s=t\leq a$, but not both;
		\item $p$ is order-preserving. That is, if $\langle a,s_0,t_0\rangle$ and $\langle a,s_1,t_1\rangle$ are both elements of $p$ and $s_0<s_1$, then $t_0<t_1$;
		\item there are no disoriented cycles in $p$ without immediate cancellation. Here, a \emph{disoriented walk} in $p$ is a sequence $\langle a_i\colon i\in n, s_i\colon i\in n+1\rangle$ such that for each $i\in n$, either $\langle a_i, s_i, s_{i+1}\rangle\in p$ or $\langle a_i, s_{i+1}, s_i\rangle\in p$. Such a walk is a \emph{cycle} if $s_0=s_n$. An immediate cancellation in such a walk is an $i\in n$ such that either $s_i\leq a_i$ or $a_i=a_{i+1}$ and $s_i=s_{i+2}$. 
	\end{enumerate}
	
\noindent $P$ is ordered by reverse extension.
\end{definition}

\noindent Note that the last item says that a disoriented cycle trivializes after removing all immediate cancellations from it.

\begin{proposition}
\label{densityproposition}
	The following sets are dense in $P$:
	
	\begin{enumerate}
		\item $D^0_{a,s}=\{p\in P\colon \exists t\ \langle a, s, t\rangle\in p\}$ for every $a, s\in\mathbb{Q}$;
		\item $D^1_{a,t}=\{p\in P\colon \exists s\ \langle a, s, t\rangle\in p\}$ for every $a, t\in\mathbb{Q}$;
		\item $D^2_{a,s,t}=\{p\in P\colon \exists n\ \exists \langle s_i\colon i\in n+1\rangle\ s_0=s$ and $\forall i\in n\ \langle a, s_i, s_{i+1}\rangle\in p$ and $s_{i+1}>t\}$ for all $a<s<t$ in $\mathbb{Q}$;
\item $D^3_{J, s, I}=\{p\in P\colon \exists a\in J\ \exists t\in I\ \langle a, s, t\rangle\in p\}$ for any $s\in \mathbb{Q}$ and nonempty rational intervals $I, J\subset\mathbb{Q}$ such that $J<s<I$.
	\end{enumerate}
\end{proposition}

\begin{proof}
	For item (1), given $p\in P$ and $a, s\in\mathbb{Q}$, there are three cases. If $s\leq a$ then add $\langle a, s, s\rangle$ into $p$. If $s>a$ and there is $t$ such that $\langle a, s, t\rangle\in p$, then leave $p$ unchanged. In the remaining case, consider the sets $\{a\}\cup\{t'\colon\exists s'<s\ \langle a, s', t'\rangle\in p \}$ and $\{t'\colon\exists s'>s\ \langle a, s', t'\rangle\in p \}$, select any value of $t$ between these two sets which is not mentioned in $p$, and add $\langle a, s, t\rangle$ into $p$. In all three cases, the resulting set is a condition in $P$, and it belongs to the set $D^0_{a,s}$: the new triple cannot stand in any disoriented cycle without inducing an immediate cancellation.
	
	Items (2, 3, 4) are similar and left to the reader.
\end{proof}

Now, let $G\subset P$ be a filter meeting all the dense sets named in the proposition. For each $a\in\mathbb{Q}$ let $\gamma_a=\{\langle s, t\rangle\colon \langle a, s, t\rangle\in\bigcup G\}$. The following is then clear from the choice of the filter $G$.

\begin{proposition}
	\label{simpleproposition}
	For every $a\in\mathbb{Q}$, $\gamma_a$ is an order-preserving permutation of $\mathbb{Q}$. In addition, if $s\leq a$ then $\gamma(s)=s$, and if $s>a$ then the set $\{\gamma_a^n(s)\colon n\in\gw\}$ is cofinal in $\mathbb{Q}$.
\end{proposition}

\noindent Let $\Gamma$ be the group of permutations of $\mathbb{Q}$ generated by the set $\{\gamma_a\colon a\in\mathbb{Q}\}$. The following is a key proposition.

\begin{proposition}
	\label{keyproposition}
The group $\Gamma$ is free over its generators. For every $\gamma\in\Gamma$, the set $\{q\in\mathbb{Q}\colon \gamma(q)=q\}$ is an initial segment of $\mathbb{Q}$.
\end{proposition}

\begin{proof}
	For the freeness, suppose that $n\in\gw$ is a number, for each $i\in n$ $a_i\in\mathbb{Q}$ and $b_i\in \{-1, 1\}$, and $\prod_i\gamma_{a_i}^{b_i}=1$; we must find $i\in n-1$ such that $a_i=a_{i+1}$ and $b_i=-b_{i+1}$. To this end, construct rationals $r_j$ for $j\in n+1$ by letting $r_0>\max_ia_i$ and $r_{j+1}>\max_i\gamma_{a_i}(r_j)$. Let $s_0\in\mathbb{Q}$ be any element greater than all $r_j$ for $j\in n+1$, and by recursion on $i\in n$ define $s_{i+1}=\gamma_{a_i}^{b_i}(s_i)$; note that $s_n=s_0$ holds. Now let $p\in G$ be a condition for which every $i\in n$ contains either the triple $\langle a_i, s_i, s_{i+1}\rangle$ or the triple $\langle a, s_{i+1}, s_i\rangle$. Then $\langle a_i\colon i\in n, s_i\colon i\in n+1\rangle$ is a disoriented cycle in $p$ and there must be an immediate cancellation in it. By the choice of $s_0$, for every $i\in n$ it must be the case that $s_i>a_i$, so that immediate cancellation means that there must be $i$ such that $a_i=a_{i+1}$ and $b_i=-b_{i+1}$ as desired.
	
	For the second sentence, let $\gamma\in\Gamma$ be any group element. By the freeness item, there is a unique number $n\in\gw$ and a product $\prod_{i\in n}\gamma_{a_i}^{b_i}$ without immediate cancellations which is equal to $\gamma$. Suppose that $r<s$ are rationals and $\gamma(s)=s$; we need to show that $\gamma(r)=r$. Consider the sequences $w_s=\langle a_i\colon i\in n, s_i\colon i\in n+1\rangle$ and $w_r=\langle a_i\colon i\in n, r_i\colon i\in n+1\rangle$ where $s_0=s$ and $s_{i+1}=\gamma_{a_i}^{b_i}(s_i)$ and similarly for the $r$ entries. Note that all $\gamma_{a_i}$ are order-preserving, so $r_i<s_i$ holds for each $i\in n+1$.
	
	Let $p\in G$ be a condition such that for each $i\in n$ it contains either $\langle a_i, s_i, s_{i+1}\rangle$ or $\langle a_i, s_{i+1}, s_i\rangle$, and similarly for the $r$ entries. Then $w_s$ is a disoriented cycle in $p$; we must show that $w_r$ is a disoriented cycle in $p$. To see this, note that the disoriented cycle $w_s$ trivializes after removing all immediate cancellations. Now, each immediate cancellation for $w_s$ is an immediate cancellation in $w_r$, so $w_r$ trivializes after removing immediate cancellations, and this is only possible if $r=r_0=r_{n+1}=\gamma(r)$.
\end{proof}

\begin{corollary}
\label{keycorollary}
An element $\delta\in\aut(\mathbb{Q})$ belongs to the closure of $\Gamma$ in $\aut(\mathbb{Q})$ if and only if for every $q\in\mathbb{Q}$ there exists $\gamma\in\Gamma$ such that $\delta(r)=\gamma(r)$ holds for all rationals $r<q$.
\end{corollary}

\begin{proof}
The right-to-left implication follows straight from the definitions. For the left-to-right implication, given $\delta$ in the closure and $q\in\mathbb{Q}$, consider the basic open neighborhood $U$ of $\delta$ consisting of all $\gamma$ such that $\gamma(q)=\delta(q)$. Since $\delta$ is in the closure of $\Gamma$, the intersection $U\cap\Gamma$. must be nonempty, containing some $\gamma$. By the proposition, if $\gb\in U\cap\Gamma$ is another element, then $\gb(r)=\gamma(r)$ for every $r<q$, since the group element $\gb^{-1}\gamma\in\Gamma$ fixes $q$ and with it all rationals $r<q$. Since $\delta$ is in the closure of $\Gamma$, it must be the case that $\delta(r)=\gamma(r)$ for all $r<q$.
\end{proof}

\noindent Now, let $X$ be a structure on $\mathbb{Q}$ in a countable language such that $\aut(X)\subseteq\aut(\mathbb{Q})$ is the closure of the group $\Gamma$. The standard construction of such a structure includes, for each $n\in\gw$, the equivalence relation $E_n$ on $\mathbb{Q}^n$ defined by $\bar x=\bar y$ if there is $\gamma\in\Gamma$ such that $\gamma(\bar x)=\bar y$, and a separate predicate for each of the equivalence classes of $E_n$. The main result of this section verifies the main properties of the group $\aut(X)$ and relates it to the discussion in \cite{rosendal:global}.

\begin{theorem}
The group $\aut(X)$ is pliable, it has a narrow basis, and it is not Roelcke precompact.
\end{theorem}

\begin{proof}
To see that the group is pliable, suppose that $U\subset\aut(X)$ is an open subgroup. Pick a finite set $b\subset\mathbb{Q}$ such that $\pstab(b)\subseteq U$. Let $\max(b)<a_0<a_1\in\mathbb{Q}$ be rational numbers. We will show that the group $V=\stab(a_1)$ has the required properties.
	
	First of all, every element of $\Gamma$ fixing $a_1$ also fixes all rationals smaller than $a_1$ by Proposition~\ref{keyproposition}; in particular, it fixes all elements of $b$. This feature persists to the closure of $\Gamma$; in particular, $V\subset U$ holds.
	
	Now, let $c\subset\mathbb{Q}$ be any finite set; we must find a conjugate of $V$ by an element of $U$ which is a subset of $\pstab(c)$. To do that, use Proposition~\ref{simpleproposition} to see that $\gamma_{a_0}\in\pstab(b)$ and there is a number $n\in\gw$ such that $\gamma_{a_0}^n(a_1)>\max(c)$. It will be enough to show that $\gamma_{a_0}^{n}V\gamma_{a_0}^{-n}\subseteq\pstab(c)$. This is immediate though from the observation that $\gamma_{a_0}^{-n}$ sends all elements of $c$ below $a_1$ and all elements of $V$ fix all rationals below $a_1$.

For the narrow basis, recall that a Polish group $\Delta$ has a narrow basis if there are open neighborhoods $U_n$ of the unit for $n\in\gw$ such that for every sequence of finite sets $F_n\subset \Delta$ for $n\in\gw$, $\Delta\neq \bigcup_nF_nU_n$. To produce a narrow basis for $\aut(X)$, let $q_n$ for $n\in\gw$ be rational numbers strictly increasing and unbounded in $\mathbb{Q}$, let $U_n=\stab(q_n)$, and work to prove that $\langle U_n\colon n>0\rangle$ is a narrow basis.

To this end, let $F_n\subset\aut(X)$ be finite sets for $n>0$. By recursion on $n\in\gw$ build rational numbers $q_n<s_n<r_n<q_{n+1}$ and elements $\delta_n\in\Gamma$ so that for every $n\in\gw$,

\begin{itemize}
\item for no $\beta\in F_{n+1}$ $\beta(s_n)=r_n$;
\item for every $m<n$, $\delta_n(s_m)=r_m$.
\end{itemize}

\noindent The construction is started with $\delta_0=1$. Now suppose that $\delta_n$ and $s_m, r_m$ for all $m<n$ have been constructed. Let $q_n<s_n<q_{n+1}$ be any rational number and $s_n<I_n<q_{n+1}$ be any nonempty interval which does not contain any of the values $\gb(s_n)$ for $\gb\in F_{n+1}$. Consider the value $u_n=\delta_n(s_n)$. By the density requirement (4) in Proposition~\ref{densityproposition} there must be a rational number $r_{n-1}<a_n<\min(\delta_n(s_n), q_n)$ such that either $\gamma_{a_n}(u_n)\in I_n$ or $\gamma_{a_n}^{-1}(u_n)\in I_n$. Accordingly, let $\delta_{n+1}=\gamma_{a_n}\delta_n$ or $\delta_{n+1}=\gamma_{a_n}^{-1}\delta_n$, and let $r_n=\delta_{n+1}(s_n)\in I_n$. This completes the recursion step.

In the end, observe that for any $m<n<k$, $\delta_n(s_m)=\delta_k(s_m)=r_m$, so the group element $\delta_n\delta_k^{-1}\in\Gamma$ fixes $s_m$ and by Proposition~\ref{keyproposition} it fixes all elements smaller than $s_m$, and $\delta_n$ and $\delta_k$ must agree on all rationals smaller than $s_m$. As the rationals $s_m$ for $m\in\gw$ form an unbounded set, the group elements $\delta_n$ for $n\in\gw$ must form a converging sequence in $\aut(\mathbb{Q})$. Let $\delta=\lim_n\delta_n\in\aut(X)$; the first demand of the recursive requirement above implies that $\delta\notin\bigcup_{n>0}F_nU_n$ as required.

For the failure of Roelcke precompactness, recall that a Polish group $\Delta$ is Roelcke precompact if for every open neighborhood $U$ of the unit there is a finite set $F\subset\Delta$ such that $\Delta=UFU$. To see why it fails for the group $\aut(X)$, choose any rational $q$ and consider the open neighborhood $U=\{\gamma\in\aut(X)\colon \gamma(q)=q\}$. Let $F$ be any finite set. By the definition of the group $\Gamma$, every $\delta\in\Gamma$ fixes an initial segment of rationals, and this feature persists to $\aut(X)$ by Corollary~\ref{keycorollary}. Thus, there must be a rational $s<q$ such that all elements of $F$ fix all rationals smaller than $s$. Then, all elements of $UFU$ fix all rationals smaller than $s$, and $\aut(X)\neq UFU$.
\end{proof}

\section{Cli groups}
\label{clisection}

Theorem~\ref{theoremB} relies on a close correspondence between a certain set-theoretic rank and a $\mathbf{\Pi}^1_1$-rank on cli groups recently isolated by Allison and Panagiotopoulos \cite{panagiotopoulos:ranks}. 

\begin{definition}
	A set $A$ has \emph{wo-rank} $0$, $\worank(A)=0$, if it is well-orderable. A set $A$ has \emph{wo-rank} $\leq\ga$, $\worank(A)\leq\ga$, if it is a union of a well-ordered collection of sets of rank $<\ga$. If no ordinal bound can be found, write $\worank(A)=\infty$.
\end{definition}

\noindent It can be proved by transfinite induction on $\ga$ that $A\subseteq B$ implies $\worank(A)\leq\worank(B)$. In consequence, the well-orderable unions in the previous definition can be taken as disjoint well-orderable unions.

\begin{definition}
	\cite{panagiotopoulos:ranks}
	Let $\Gamma$ be a Polish non-archimedean group. A pair $(V, U)$ of open neigborhoods of the identity in $\Gamma$ has \emph{cli-rank} $0$, $\clirank(V, U)=0$, if $U\subseteq V$, and cli-rank $\leq\ga$, $\clirank(V, U)\leq\ga$, if there is an open neighborhood $W\subseteq U$ of the identity such that for every $\gamma\in U$, $(V, \gamma W\gamma^{-1})$ has rank $<\ga$. If no ordinal bound can be found, write $\clirank(U, V)=\infty$.
\end{definition}

\noindent An immediate transfinite induction argument shows that $V'\subseteq V$ and $U'\subseteq U$ together imply $\clirank(V', U)\geq\clirank(V, U')$. If $U$ happens to be a subgroup of $\Gamma$ and $V\subseteq U$, then the evaluation of $\clirank(V, U)$ does not really depend on $\Gamma$, but only on $U$. We are now ready to prove Theorem~\ref{theoremB}.

\begin{theorem}
	Let $X$ be a countable structure. The following are equivalent:
	
	\begin{enumerate}
		\item the group $\aut(X)$ is cli;
		\item in $W[[X]]$, every set appears at some stage of the well-orderability hierarchy.
	\end{enumerate}
\end{theorem}

\begin{proof}
Let $X$ be a countable structure, let $\aut(X)$ be the group of its automorphisms acting on $X$ by application, let $I$ be the ideal of finite subsets of $X$. Let $W[[X]]$ be the permutation model associated with the dynamical ideal $\aut(X)\acts X, I$. The parallel between the two ranks is elucidated by the following two propositions.

\begin{claim}
	\label{first}
	Let $B\in W[[X]]$ be a set, let $U\subset\aut(X)$ be an open subgroup. If $\clirank(\stab(B), U)\leq\ga$ then $\worank(U\cdot B)\leq\ga$.
\end{claim}

\begin{proof}
	By transfinite induction on $\ga$, for all $B, U$ simultaneously. If $\clirank(\stab(B), U)=0$ then $U\subseteq\stab(B)$, so $U\cdot B=\{B\}$, $U\cdot B$ is therefore well-orderable, and it has wo-rank $0$. Suppose that $\clirank(\stab(B), U)\leq\ga$ and find a subgroup $W\subseteq U$ such that for all $\gamma\in U$, $\clirank(\stab(B), \gamma W\gamma^{-1})<\alpha$. Divide $U\cdot B$ into $W$-orbits. The set of all $W$-orbits which are subsets of $U\cdot B$ is in the model $W[[X]]$ and well-orderable there, as each of these orbits is $W$-invariant. So, to show that $\worank(U\cdot B)\leq\ga$, it is enough to show that each of these $W$-orbits has wo-rank $<\ga$. Pick $\gamma\in U$. Then, $\stab(\gamma\cdot B)=\gamma\stab(B)\gamma^{-1}$, and $\clirank(\gamma\stab(B)\gamma^{-1}, W)=\clirank(\stab(B), \gamma^{-1}W\gamma)<\ga$. By the induction hypothesis, $\worank(W\cdot \gamma\cdot B)<\ga$ and the proof is complete.
\end{proof}

\begin{claim}
	\label{second}
	Let $A\subset X$ be a set of wo-rank $\leq\ga$. There is an open subgroup $U\subset\aut(X)$ such that for every finite set $a\subset A$, $\clirank(\pstab(a), U)\leq\ga$.
\end{claim}

\begin{proof}
	By transfinite induction on $\ga$. If $\ga=0$, then $A$ is well-orderable, the group $U=\pstab(A)$ is open, and it works: for each finite set $a\subset A$, $\pstab(A)\subseteq\pstab(a)$, so $\clirank(\pstab(a), U)=0$. Now, suppose that $\worank(A)\leq\ga$, and let $A=\bigcup B$, where $B$ is a well-orderable set of sets of wo-rank $<\ga$. We claim that $U=\pstab(B)$ works. Let $a\subset A$ be a finite set. Let $C\subset B$ be a finite subset of $B$ such that $a\subset\bigcup C$. By the definitions, there is an ordinal $\gb<\ga$ such that $\worank(\bigcup C)\leq\gb$. Let $W\subset U$ be an open subgroup witnessing the induction hypothesis for $\bigcup C$; We claim that it witnesses the fact that $\clirank(\pstab(a), U)\leq\ga$. Let $\gamma\in U$ be any element. Then $C$ is fixed by $\gamma$, so $\gamma\cdot a\subset\bigcup C$. By the choice of $W$, $\clirank(\pstab(\gamma\cdot a), W)\leq\gb$. However, $\pstab(\gamma\cdot a)=\gamma\pstab(a)\gamma^{-1}$, so $\clirank(\pstab(\gamma\cdot a), W)=\clirank(\pstab(a), \gamma^{-1}W\gamma)<\ga$ as desired.
\end{proof}

\noindent The theorem now quickly follows. By \cite[Theorem 1.2]{panagiotopoulos:ranks}, a Polish group $\Gamma$ is cli iff for every open subset $U\subset \Gamma$, $\clirank(U, \Gamma)\neq\infty$; in fact, in such a case these ranks are all countable.  So, suppose first that $\aut(X)$ is cli, and let $A\in W[[X]]$ be any set; we must show that $\worank(A)\neq\infty$. By the definition of the model $W[[X]]$, there is a finite set $a\subset X$ such that $\pstab(a)\subseteq\stab(A)$. By \cite[Theorem 1.2]{panagiotopoulos:ranks} and the monotonicity properties of cli-rank, $\clirank(\stab(A), \pstab(a))\neq\infty$ holds. By Claim~\ref{first}, the wo-rank of the set $\pstab(a)\cdot A=A$ is not infinite. 

Suppose on the other hand that for every set $A\in W[[X]]$, $\worank(A)\neq\infty$. In particular, $\worank(X)\neq\infty$, and by Claim~\ref{second}, there is an open subgroup $U\subset\aut(X)$ such that for every finite set $a\subset X$, $\clirank(\pstab(a), U)\leq\ga$. The monotonicity properties of cli-rank together with \cite[Theorem 1.2]{panagiotopoulos:ranks} show that the open subgroup $U$ of $\aut(X)$ is cli. It is easy to argue abstractly that $\Gamma$ is cli: if $d$ is a complete left-invariant metric on $U$ bounded by $1$, then it can be extended to a complete left-invariant metric on $\Gamma$ by setting $d(\gamma, \delta)=d(1, \gamma^{-1}\delta)$ if $\gamma^{-1}\delta\in U$ and $d(\gamma, \delta)=2$ otherwise. Thus, $\Gamma$ is cli as needed. This completes the argument.
\end{proof}

\section{Classification of open subgroups}
\label{bootstrappingsection}

In the main result of this section, we find a necessary and sufficient amalgamation criterion for when a limit of a relational Fraiss{\' e} class with disjoint amalgamation has classification of open subgroups.

\begin{definition}
	Let $\mathcal{F}$ be a relational Fraiss{\'e} class.
	
	\begin{enumerate}
		\item Structures $A, B\in\mathcal{F}$ are in \emph{amalgamation position} if $A\restriction (\dom(A)\cap \dom(B)) =
		B\restriction (\dom(A)\cap\dom(B)$;
		\item for a pair $A, B\in\mathcal{F}$ in amalgamation position, a \emph{minimal amalgamation} of $A$ and $B$ is a structure $C$ on $\dom(A)\cup\dom(B)$ such that $C\restriction\dom(A)=A$ and $C\restriction \dom(B)=B$.
	\end{enumerate}
\end{definition} 

\begin{definition}
	\label{bootstrappingdefinition}
	Let $\mathcal{F}$ be a relational Fraiss{\' e} class. Let $A, B\in\mathcal{F}$ be a pair in amalgamation position and $C, D$ be minimal amalgamations of $A$ and $B$.
	
	\begin{enumerate}
		\item $C, D$ are \emph{close} if there is a \emph{focus of difference} between $C$ and $D$--an element $b\in\dom(B)\setminus\dom(A)$ such that $C\restriction (\dom(A)\cup\dom(B))\setminus\{b\}$ is equal to $D\restriction (\dom(A)\cup\dom(B))\setminus\{b\}$;
		\item a \emph{walk} from $C$ to $D$ is a sequence $\langle C_i\colon i\in j+1\rangle$ of minimal amalgamations of $A$ and $B$ such that $C_0=C$, $C_j=D$, and for each $i\in j$, $C_i$ and $C_{i+1}$ are close;
		\item $\mathcal{F}$ is a \emph{bootstrapping class} if it has disjoint amalgamation and for any pair $A, B$ in amalgamation position and any two minimal amalgamations $C,D$ of $A$ and $B$ there is a walk from $C$ to $D$.
	\end{enumerate}
\end{definition}

\noindent Note that the concept of closeness as stated is not symmetric in that the focus of difference is always on the $B$-side. In the walks, the foci are going to change, but they always must remain on the same side. Most proofs showing that a certain relational Fraiss{\' e} class is a bootstrapping class
are quite natural: among all minimal amalgamations of $A$ and $B$, they identify
an optimal one and then show how an arbitrary minimal amalgamation can be
transformed to the optimal one one point at a time. 

To kick off the list of examples, recall that a relational Fraiss{\' e} class $\mathcal{F}$ is \emph{hereditary} if it is closed under subsets in the sense that if $A, B$ are structures for its relational language with the same domain, $A\in\mathcal{F}$, and all relations in $B$ are subsets of the corresponding relations in $A$, then $B\in\mathcal{F}$ holds.

\begin{example}
	\label{hereditaryexample}
	If $\mathcal{F}$ is a relational, hereditary Fraiss{\' e} class in finite language with disjoint amalgamation, then $\mathcal{F}$ is a bootstrapping class.
\end{example}

\begin{proof}
	Let $A,B$ be a pair of structures in $\mathcal{F}$ in amalgamation position. Let $D$ be the
	minimal disjoint amalgamation of $A$ and $B$ obtained by taking the union of
	the corresponding relations in $A$ and $B$. This is indeed a structure in $\mathcal{F}$ by the heredity assumption, since its relations are subsets of any minimal disjoint amalgamation of $A$ and $B$.
	Now, if $C$ is any minimal disjoint amalgamation of $A$ and $B$, it is possible to
	produce a walk from $C$ to $D$ by setting $C_0=C$ and then erasing tuples in relations in $C$ which are not in $D$ one by one; any element of $\dom(B)\setminus\dom(A)$ which is in the tuple being erased will be the focus of difference between two consecutive structures in the walk.  It is easy to see that all structures obtained in this walk are in $\mathcal{F}$ as $\mathcal{F}$ is hereditary. By the finiteness of the language of $\mathcal{F}$, the walk will end in $D$ after finitely many steps in the walk. The proof is complete.
\end{proof} 

\begin{example} 
	The class of all graphs is bootstrapping.
\end{example}

\begin{example}
	For any number $k>2$, the class of all graphs not containing a
	clique of cardinality $k$ is bootstrapping.
\end{example}

\noindent The non-hereditary bootstrapping classes are much more interesting.

\begin{example}
	\label{orderexample}
	The class of linear orderings is bootstrapping.
\end{example}

\begin{proof}
	Let $A$ and $B$ be a pair of linear orderings in amalgamation position. There is an optimal minimal amalgamation $D$ of $A$ and $B$, in which if $a\in\dom(A)\setminus\dom(B)$ and $b\in\dom(B)\setminus\dom(A)$ are consecutive points, then $b\leq a$ holds. There
	is only one amalgamation satisfying this property: inside every interval specified by two consecutive points of $A\cap B$, elements of $\dom(B)\setminus\dom(A)$ come before elements of $\dom(A)\setminus\dom(B)$. Now, given an arbitrary minimal disjoint
	amalgamation $C$ of $A$ and $B$, it is possible to produce a walk from $C$ to $D$ by at each step flipping some pair of two consecutive elements $a\in\dom(A)\setminus\dom(B)$ and $b\in\dom(B)\setminus\dom(A)$ for which $a\leq b$ holds. Note that $b$ will always be the focus of difference of two such amalgamations on the walk.
	
	The walk has to end after finitely many steps, since for every element $a\in A$ the set of all elements of $b\in B$ which are above $a$ is 
	non-increasing along the walk, and at each step of the walk one of these finite sets actually becomes smaller. The final element of the walk must be the amalgamation $D$.
\end{proof}

\begin{example}
	\label{posetexample}
	The class of partial orderings is bootstrapping.
\end{example}

\begin{proof}
	Let $A$ and $B$ be a pair of partial orderings in amalgamation position. There is an optimal amalgamation $D$ of $A$ and $B$ in which elements $a\in A$ and $b\in B$ are comparable only if there is an element of $a\cap b$ between them. This is in fact the inclusion-smallest minimal amalgamation available. 
	
	Now, given an arbitrary minimal amalgamation $C$ of $A$ and $B$, construct a walk from $C$ by demanding that at each stage of the walk exactly one pair $\langle a, b\rangle$ or $\langle b, a\rangle$ is removed from the amalgamation for $a\in A\setminus B$ and $b\in B\setminus A$, if at all possible. Successive amalgamations on such a walk are clearly close, because the point $b$ will be the focus of difference. 
	
	The walk must end after finitely many steps as in each step the cardinality of the amalgamating order relation is decreased. It is enough to show that the final station of the walk must be equal to $D$; this amounts to showing that if $E\neq D$ is any amalgamation, then the walk can proceed one more step past $E$. Indeed, if $E\neq D$, then there must be elements $a\in A\setminus B$ and $b\in B\setminus A$ which are comparable in $E$ and such that there is no element of $A\cap B$ between them. Selecting such points $a, b$ such that the set of elements between them is inclusion-minimal, it becomes clear that such points $a, b$ have no element of $A\cup B$ strictly between them, and then the pair $\langle a, b\rangle$ or $\langle b, a\rangle$ can be removed from the relation and the resulting relation will remain a partial ordering and an amalgamation of $A$ and $B$.
\end{proof}

\begin{example}
	The class of rational metric spaces is bootstrapping.
\end{example} 

\begin{proof}
	Let $A$ and $B$ be a pair of rational metric spaces in amalgamation position, with
	metrics $d_A$ and $d_B$. There is an optimal minimal amalgamation $D$ of $A$ and $B$, in which if $a\in \dom(A)\setminus\dom(B)$ and $b\in\dom(B)\setminus\dom(A)$ are points
	then the distance of $a$ and $b$ is the minimum of the set $\{d_A(a, c) + d_B(c, b)\colon c\in
	\dom(A)\cap \dom(B)\}$. In view of the triangle inequality, this is the largest metric
	available. Now, given an arbitrary minimal amalgamation $C$ of $A$ and $B$. Construct a walk from $C$ such that at each step, exactly one distance between an element of $\dom(A)\setminus\dom(B)$ and an element of $\dom(B)\setminus\dom(A)$ is increased as much as possible without violating the triangle inequality.
	
	The walk has to end after finitely many steps. To see this, select a positive rational number $q$ such that every distance used in $C$ is an integer multiple of $q$. It is easy to prove by induction that all distances in the successive amalgamations on the walk will be integer multiples of $q$. Thus, the sum of all distances used in the amalgamation is increasing each time by at least $q$, and it cannot overtake the sum of all distances in $D$. It will be enough to show that the final amalgamation of the walk must be $D$. 
	
	To do this, it is enough to show that if $E$ is an amalgamation in which all distances are integer multiples of $q$ and at least one distance is smaller than that in $D$, then adding $q$ to one of the distances will not violate the triangle inequality.
	
	Find points $a\in\dom(A)\setminus\dom(B)$ and $b\in\dom(B)\setminus\dom(A)$ such that $d_E(a, b) < d_D(a, b)$ and $d_E(a, b)$ is the smallest possible. It will be enough to show that for every $c\in\dom(A)\cup\dom(B)$, $d_E(a, b)+q\leq d_E(a, c)+d_E(c, b)$. Without loss assume that $c\in\dom(A)$. Since all distances in $E$ are integer multiples of $q$, it will be enough to derive a contradiction from the assumption
	that $d_E(a, b) = d_E(a, c) + d_E(c, b)$. 
	
	Now, if $c\in\dom(A)\cap\dom(B)$, this would mean that $d_E(a, b)\geq d_D(a, b)$, contrary to the choice of the points $a$ and $b$. If, on the other hand, $c\in\dom(A)\setminus\dom(B)$ holds, then by the minimal choice of the points $a$ and $b$ it has to be the case that $d_E(c, b) = d_D(c,b)$. Pick an element $e\in\dom(A)\cap\dom(B)$ such that $d_E(c, b) = d_A(c, e) + d_B(e, b)$. By the triangle inequality in $A$, $d_A(a, e)\leq d_A(a, c)+d_A(c, e)$. The strict inequality is impossible here, since
	it would violate the triangle inequality in $E$ in the triangle with vertices $a, b, e$. However, the equality leads to the conclusion that $d_E(a, b) = d_A(a, e)+d_B(e, b)\geq d_D(a, b)$, violating the choice of the points $a, b$.
\end{proof}

\begin{example}
	\label{selectorexample}
	The class of selectors is bootstrapping.
\end{example}

\noindent This is the class of all finite structures $A$ equipped with a function $f_A$ which, to each nonempty subset of $\dom(A)$ assigns one of its elements. It is not diffcult to restate this class as one with an infinite relational language.

\begin{proof}
	The argument is slightly different from the previous ones in that one does not need to find an optimal amalgamation to streamline the argument. Let
	$A$ and $B$ be a pair in amalgamation position and $C$ and $D$ two of its minimal disjoint
	amalgamations. To get a walk from $C$ to $D$, at each amalgamation in the walk switch the selector value on a single nonempty set $a\subset\dom(A)\cup\dom(B)$ to $f_D(a)$.
	
	Observe that this set $a$ must have nonempty intersection with $\dom(B)\setminus\dom(A)$ (otherwise its value would be dictated by $A$) and every point of this intersection will be a focus of difference between the two successive amalgamations on the walk. It is clear that the walk must terminate after finitely many steps in the amalgamation $D$.
\end{proof}

\noindent The collection of bootstrapping Fraiss{\' e} classes is closed under the operation of superposition. Here, if $\mathcal{F}$ and $\mathcal{G}$ are Fraiss{\' e} classes in disjoint relational languages $\mathcal{L}, \mathcal{K}$ and with disjoint amalgamation, write $\mathcal{F}*\mathcal{G}$ for the \emph{superposition} of $\mathcal{F}$ and $\mathcal{G}$, the class of all finite structures $A$ such that $A\restriction \mathcal{L}\in \mathcal{F}$ and $A\restriction\mathcal{K}\in \mathcal{G}$. The result is again a Fraiss{\' e} class with disjoint amalgamation. The following example is simple and at the same time very useful.

\begin{example}
	\label{superpositionexample}
	If $\mathcal{F}$ and $\mathcal{G}$ are bootstrapping Fraiss{\' e} classes in disjoint languages $\mathcal{L}, \mathcal{K}$, then so is their superposition $\mathcal{F}*\mathcal{G}$.
\end{example}

\begin{proof}
	Let $A$ and $B$ be structures in $\mathcal{F}*\mathcal{G}$ in amalgamation position, and let $C, D$ be minimal amalgamations. Produce a walk from $C$ to $D$ which is an amalgamation of two parts. In the first part, the difference between $C$ and $D$ in $\mathcal{L}$ is gradually rectified while the restriction to $\mathcal{K}$ is kept the same; in the second part, the difference between $C$ and $D$ in $\mathcal{K}$ is gradually rectified while the restriction to $\mathcal{L}$ is kept the same. 
\end{proof}

\begin{example} 
	Let $\mathcal{F}$ be the Fraiss{\' e} class of equivalence relations. Then $\mathcal{F}$ is not bootstrapping. To see that, consider disjoint structures $A, B\in\mathcal{F}$, both of them consisting of two equivalent elements. There are two possible minimal disjoint amalgamations of $A$ and $B$, but one cannot be obtained from the other by just switching the equivalence on a single point.
\end{example}

\begin{example}
	\label{ultraexample}
	Let $\mathcal{F}$ be the Fraiss{\' e} class of ultrametric spaces in which all distances are rational. Then $\mathcal{F}$ is not bootstrapping. To see that, let $A, B$ be two disjoint metric spaces, each consisting of two points of distance $1$. Consider the amalgamations $C_2$ and $C_3$ obtained by assigning distance $2$ or $3$ between points of $A$ and $B$ respectively. There is no way to adjust the amalgamation $C_2$ (or $C_3$) in just one point, so there is no walk from $C_2$ to $C_3$.
\end{example}

\noindent The main result of this section is the following characterization of classification of open subgroups, including Theorem~\ref{theoremC} as its sub-statement. To set up the notation, let $\mathcal{F}$ be a relational Fraiss{\' e} class with disjoint amalgamation and $X$ its limit. For each structure $B\in\mathcal{F}$ let $X^B$ be the set of all injective maps from $B$ to $X$ which preserve all relations in the language of $\mathcal{F}$. For all structures $A\subseteq B$ consider the equivalence relation $E_A$ connecting two elements of $X^B$ if they agree on $A$, and a graph $G_A$ connecting two elements of $X^B$ if they agree on all elements of $B$ except for one point, and this point does not belong to $A$. Note that the group $\aut(X)$ acts on $X^B$ by composition, and this action preserves both $E_A$ and $G_A$. Note also that $G_A\subset E_A$ holds; the key issue is if $E_A$ is the connectedness relation of $G_A$ or not.

\begin{theorem}
	\label{bootstrappingtheorem}
	Let $\mathcal{F}$ be a relational Fraiss{\' e} class with disjoint amalgamation and let $X$ be its limit. The following are equivalent:
	
	\begin{enumerate}
		\item $\mathcal{F}$ is bootstrapping;
		\item for all structures $A\subseteq B$ in the class $\mathcal{F}$, the connectedness equivalence relation of $G_A$ is equal to $E_A$;
		\item $X$ admits classification of open subgroups.
	\end{enumerate}	
\end{theorem}

\begin{proof}
	To show that (1) implies (2), suppose that the class $\mathcal{F}$ is bootstrapping. Suppose that $A\subseteq B$ are structures in $\mathcal{F}$. First, we introduce some notation. Two $E_A$-related maps $\phi, \psi\in X^B$ are \emph{disjoint over $A$} if $\rng(\phi)\cap\rng(\psi)=\rng(\phi\restriction A)$. For any two such maps the symbol $\phi\dot\cup\psi$ describes a map from two copies $B_0$ and $B_1$ of $B$ which intersect in $A$; the pullback from $X$ by $\phi\dot\cup\psi$ is then a minimal disjoint amalgamation of $B_0, B_1$.
	
	\begin{claim}
		For any $\phi\in X^B$ and any minimal disjoint amalgamation $D$ of $B_0$ and $B_1$ there is a $G_A$-walk from $\phi$ to some $\psi\in X^B$, disjoint over $A$ with $\phi$, and such that the pullback from $X$ by $\phi\dot\cup\psi$ is $D$.
	\end{claim}
	
	\begin{proof}
First use the universality properties of $X$ to walk from $\phi$ to some $\psi_0$ which is disjoint over $A$ from $\phi$. Consider the amalgamation $C$ of $B_0\cup B_1$ which is a pullback from $X$ by $\phi\cup\psi_0$. There is an amalgamation walk $\langle C_i\colon i\in n+1\rangle$ from $C$ to $D$. By the universality properties of $X$, this can be translated to a $G_A$-walk $\langle \phi_i\colon i\in n+1\rangle$ such that all of these maps are disjoint over $A$ with $\phi$ and the pullback from $X$ by $\phi\dot\cup\psi_i$ is equal to $C_i$. Then $\psi=\psi_n$ is as desired.
	\end{proof} 
	
	\begin{claim}
		For any $\phi, \psi\in X^B$ which are $E_A$-related and disjoint over $A$ there is a $G_A$-walk from $\phi$ to $\psi$.
	\end{claim}
	
\begin{proof}
	By the previous claim, find a walk from $\phi$ to some $\psi'$ which is disjoint from it over $A$ such that the pullback amalgamations by $\phi\dot \cup\psi$ and $\phi\dot\cup\psi'$ are the same. Use the homogeneity properties of $X$ to find an automorphism $\delta\in\aut(X)$ such that $\delta$ extends the map $(\phi\dot\cup\psi)(\phi\dot\cup\psi')^{-1}$. Then $\delta$ transports the walk from $\phi$ to $\psi'$ to a walk from $\phi$ to $\psi$.
\end{proof}	

\begin{claim}
For any $\phi_0, \phi_1\in X^B$ which are $E_A$-related there is a $G_A$-walk from $\phi_0$ to $\phi_1$.
\end{claim}	

\begin{proof}
	Just use the homogeneity properties of $X$ to find a map $\psi\in X^B$ which is $E_A$-related to both $\phi_0$ and $\phi_1$, and disjoint from both $\phi_0$ and $\phi_1$ over $A$. Then use the previous claim and find a walk from $\phi_0$ to $\psi$ and a walk from $\psi$ to $\phi_1$. Concatenating the two completes the proof.
\end{proof}

\noindent (1)$\to$(2) follows.

	To show that (2) implies (1), suppose that (2) holds and $A, B$ are structures in $\mathcal{F}$ in amalgamation position and $C, D$ are their amalgamations. By the universality properties of $X$, there are maps $\pi\in X^A$ and $\phi, \psi\in X^B$ which agree on $A\cap B$ and such that $C$ is the pullback from $X$ by $\pi\cup\phi$ and $D$ is the pullback from $X$ by $\pi\cup\psi$. In addition, it is possible to select $\phi$ and $\psi$ so that their ranges intersect with $\rng(\pi)$ only in $\rng(\pi\restriction A\cap B)$ . By (2), there is a $G_{A\cap B}$-walk $\langle \psi_i\colon i\in n+1\rangle$ such that $\psi_0=\phi$ and $\psi_n=\psi$. Using the homogeneity and universality features of $X$, such a walk can be found such that for every point $x$ appearing in one of its embeddings, if $x\in\rng(\pi)$ then $x\in\rng(\pi\restriction A\cap B)$. For each $i\in n+1$, let $C_i$ be the amalgamation of $A$ and $B$ which is the pullback from $X$ using the inverse of the map $\pi\cup\psi_i$. A review of the definitions shows that $\langle C_i\colon i\in n+1\rangle$ is a walk of amalgamations from $C$ to $D$.

	To show that the negation of (2) implies the negation of (3), write $F_A$ for the connectedness relation of $G_A$ and assume that (2) fails. Find structures $A\subseteq B$ in $\mathcal{F}$ and an embedding $\phi\in X^B$ such that $[\phi]_{F_A}$ is a proper subset of $[\phi]_{E_A}$. Write $a=\rng(\phi\restriction A)\subset X$ and consider the set $\Delta\subseteq\aut(X)$ defined by $\gamma\in\Delta$ iff $\gamma\in\pstab(a)$ and $\gamma$ permutes $[\phi]_{F_A}$. It will be enough to show that $\Delta$ is an open subgroup of $\aut(X)$ violating classification of open subgroups.
	
	First of all, it is clear from the definitions that $\Delta$ is in fact a subgroup of $\aut(X)$. Second, action by elements of $\pstab(a)$ preserves the graph $G_A$ on elements of $X^B$ which extend $\phi\restriction A$, so membership $\gamma\in\Delta$ is verified by checking that the value of $\gamma\cdot\phi$ belongs to $[\phi]_{F_A}$. In consequence, the group $\Delta\subseteq\aut(X)$ is open. Now suppose that $c\subset X$ is a finite set and work to prove that $\pstab(c)\subseteq\Delta\subseteq\stab(c)$ fails. There are two cases. Either $c\subseteq a$. In this case, use the initially assumed failure of (2) to find an element $\psi\in X^B$ which is $E_A$-related but not $F_A$-related to $\phi$, and an automorphism $\gamma\in\aut(X)$ extending the map $\psi\phi^{-1}$. Then $\gamma\in\pstab(a)\subseteq\pstab(c)$ holds, and obviously $\gamma\in\Delta$ fails. Or $c\setminus a\neq 0$. In this case, let $x\in X$ be an element of the set difference. Use the ultrahomogeneity properties of $X$ to find an element $\psi\in [\phi]_{F_A}$ such that $x\notin\rng(\psi)$. Use the ultrahomogeneity properties of $X$ again to find an element $\gamma\in\aut(X)$ such that $\gamma\in\pstab(\rng(\psi))$ and $\gamma(x)\notin c$. Clearly, $\gamma\in\stab(c)$ fails. At the same time, $\gamma\in\Delta$ must hold, since $\gamma$ fixes an element (namely $\psi$) of the equivalence class $[\phi]_{F_A}$ which is $\pstab(a)$-invariant.
	
	To show that (2) implies (3), assume that (2) holds. The following claim familiar from \cite[Theorem 4.2.9]{hodges:book} is key:
	
	\begin{claim}
		\label{generationclaim}
		For any two finite sets $a, b\subset X$, $\pstab(a\cap b)$ is equal to the group generated by $\pstab(a)\cup\pstab(b)$.
	\end{claim}
	
	\begin{proof}
		The right-to-left inclusion is clear. For the opposite inclusion, note that both groups are open and therefore closed, so it is enough to show that the group generated by $\pstab(a)\cup\pstab(b)$ is dense in $\pstab(a\cap b)$. Let $\gamma\in\pstab(a\cap b)$ be any element and $c\subset X$ be any finite superset of $a\cup b$; we must produce a finite composition $\delta$ of elements of $\pstab(a)\cup\pstab(b)$ such that $\gamma\restriction c=\delta\restriction c$.
		
		To this end, choose structures $A\subseteq B$ in $\mathcal{F}$ and a map $\phi\in X^B$ such that $a\cap b=\rng(\phi\restriction A)$ and $c=\rng(\phi)$. Note that $\gamma\cdot \phi$ is $E_A$-related to $\phi$. Use the assumption (2) to find a finite walk $\langle \phi_i\colon i\in n+1\rangle$ in the graph $G_A$ from $\phi$ to $\gamma\cdot\phi$. For each $i\in n$ let $\gamma_i\in\aut(X)$ be any element extending $\phi_i^{-1}\phi_{i+1}$ witness the graph relation. Observe that $\gamma_i\in\pstab(a\cap b)$ and there is a single point $x_i\in\rng(\phi_i)$ which is not in fixed by $\gamma_i$. It will be enough to show that for each $i\in n$, $\gamma_i$ is a composition of some elements of $\pstab(a)\cup\pstab(b)$, since the composition $\delta$ of $\gamma_i$'s in decreasing order has the required property.
		
		This is proved by induction on $i$. Suppose that the hypothesis has been verified for all $j\in i$, and write $\delta_i$ for the composition of all $\gamma_j$ for $j\in i$ in decreasing order. Consider the group element $\eps=\delta_i^{-1}\gamma_i\delta_i$. Then $\eps$ fixes all entries of $a\cup b$ except possibly the point $\delta_i^{-1}(x_i)$. This point cannot be in $a\cap b$ as the $\gamma_j$'s are all in $\pstab(a\cap b)$; so it must be in at most one of $a$ or $b$. Accordingly, $\eps$ belongs to $\pstab(b)$ or $\pstab(a)$. In conclusion, $\gamma_i=\delta_i\eps\delta_i^{-1}$ is a composition of some elements in $\pstab(a)\cup\pstab(b)$ as desired.
	\end{proof}
	
	\noindent Now, suppose that $\Delta\subseteq\aut(X)$ is an open subgroup. By the claim, the nonempty set of all finite sets $a$ such that $\pstab(a)\subseteq\Delta$ is closed under intersection. In conclusion, there must be an inclusion-smallest set $a\subset X$ such that $\pstab(a)\subseteq\Delta$. It will be enough to show that all elements of $\Delta$ permute the set $a$; for then, $\pstab(a)\subseteq\Delta\subset\stab(a)$ holds.
	
	Let $\gamma\in\Delta$ be arbitrary. Then $\pstab(\gamma\cdot a)\subset\Delta$ must hold: if $\delta\in\pstab(\gamma\cdot a)$ then $\gamma^{-1}\delta\gamma\in\pstab(a)$, so $\gamma^{-1}\delta\gamma\in\Delta$ and $\delta\in\Delta$. By the smallest choice of $a$ it follows that $a=\gamma\cdot a$ as desired.
\end{proof}

\noindent There are many other countable structures which admit classification of open subgroups. The following propositions describe two classes of such structures which are ubiquitous in the literature on permutation models.

\begin{proposition}
Let $X$ be a countable structure in which each $\aut(X)$-orbit is finite. Then $X$ admits classification of open subgroups.
\end{proposition}

\begin{proof}
Let $F$ be the family of all finite subsets of $X$ invariant under the action of $\aut(X)$, and suppose that $X=\bigcup F$. Let $\Delta\subseteq\aut(X)$ be an open subgroup. Find a finite set $b\subset X$ such that $\pstab(b)\subseteq\Delta$ holds. Using the initial assumption on the structure $X$, we  may enlarge the finite set $b$ if necessary to one which is $\aut(X)$-invariant. Then $\Delta\subset\stab(b)=\aut(X)$, so $\pstab(b)\subseteq\Delta\subseteq\stab(b)$ holds as desired for classification of open subgroups.
\end{proof}

\begin{definition}
Let $X$ and $Y$ be relational structures in respective disjoint languages $\mathcal{L}, \mathcal{K}$. Then $X\times Y$ denotes the structure whose language is $\mathcal{L}\cup\mathcal{K}$ with one additional binary symbol $E$ where for every relational symbol $R\in\mathcal{L}$, $\bar z\in R^{X\times Y}$ if the tuple of $X$-coordinates of the tuple $\bar z$ belongs to $R^X$, for every relational symbol $S\in\mathcal{K}$, $\bar z\in S^{X\times Y}$ if all elements on $\bar z$ share the same $X$-coordinate and the tuple of the $Y$-coordinates belongs to $S^Y$, and $E$ connects any pair of points with the same $X$-coordinate.
\end{definition}

\begin{definition}
A structure $X$ is \emph{locally finite} if for every finite set $a\subset X$, the \emph{algebraic closure} of $a$, the set $\{x\in X\colon$ the $\pstab(a)$-orbit of $x$ is finite$\}$ is finite.
\end{definition}

\begin{proposition}
If $X$ and $Y$ are relational structures and $X$ is locally finite and has classification of open subgroups and $Y$ is finite, then $X\times Y$ admits classification of open subgroups.
\end{proposition}

\begin{proof}
First observe that the function $\pi\colon\aut(X\times Y)\to\aut(X)$ defined by $\pi(\gamma)(x_0)=x_1$ if there are $y_0, y_1\in Y$ such that $\gamma(x_0, y_0)=\langle x_1, y_1\rangle$ is a continuous and open homomorphism. The standard diagram-chasing argument is left to the reader.

Now, suppose that $\Delta\subset\aut(X\times Y)$ is an open subgroup. Write $\Gamma=\pi''\Delta$; this is an open subgroup of $\aut(X)$. By the assumption, there is a finite set $a\subset X$ such that $\pstab(a)\subseteq\Gamma\subseteq\stab(a)$. Replacing $a$ with its algebraic closure if necessary, we may assume that $a$ is algebraically closed. It will be enough to show that $\pstab(a\times Y)\subseteq\Delta\subseteq\stab(a\times Y)$ holds.

It is clear that $\Delta\subseteq\stab(a\times Y)$ holds: for any element $\delta\in\Delta$, $\pi(\delta)\in\Gamma\subseteq\stab(a)$ holds, therefore $\delta$ must permute the set $a\times Y$. To show that $\pstab(a\times Y)\subseteq\Delta$ holds, it is enough to prove that $\Delta$ is dense in $\pstab(a\times Y)$. To this end, let $\delta\in\pstab(a\times Y)$ be an arbitrary element of $\aut(X\times Y)$, and let $b\subset X$ be an arbitrary finite set; we must show that there is an element $\delta'\in\Delta$ such that $\delta\restriction b\times Y=\delta'\restriction b\times Y$.

To this end, first observe that $\pi(\delta)\in\pstab(a)\subset\Gamma$, so there is $\gb_0\in\Delta$ such that $\pi(\gb_0)=\pi(\delta)$. Let $c\subset X$ be a finite set such that $\pstab(c\times Y)\subset\Delta$ holds. Let $\gamma\in\pstab(a)\subseteq\aut(X)$ be any element such that $c\cap \gamma\cdot b=0$. Observe that $\gamma\in\pstab(a)\subseteq\Gamma$, so there is $\gb_1\in\Delta$ such that $\pi(\gb_1)=\gamma$. Now, let $\gb_2$ be the element of $\aut(X\times Y)$ which fixes all elements of $X\times Y$ except for every $x\in b$ and every $y\in Y$ it sends $\gb_1(x, y)$ to $\gb_1\gb_0^{-1}\delta(x, y)$. Note that $\gb_2\in\pstab(c\times Y)$ holds, so $\gb_2\in\Delta$ must hold as well. Now consider the element $\delta'=\gb_0\gb_1^{-1}\gb_2\gb_1$. Clearly, $\delta'\in\Delta$ holds, and a diagram-chasing argument together with the definition of $\gb_2$ shows that $\delta'\restriction b\times Y=\delta\restriction b\times Y$.
\end{proof}

\noindent There is a great number of other structures which admit classification of open subgroups because they have the strong small index property, a feature much more challenging to check than the mere classification. The list of such structures contains among others the countable atomless Boolean algebra  \cite{truss:smallindex} and the infinitely branching tree with a root \cite{moller:trees}.

\section{Consequences of classification}
\label{finalsection}

In this section, we prove several consequences of classification of open subgroups. We start with a brief argument for Theorem~\ref{theoremD}; it is a more precise version of \cite[Theorem 4.2]{blass:svc}.

\begin{theorem}
	If $X$ admits classification of open subgroups, then in $W[[X]]$, $\hf(X)$ is a support set.
\end{theorem}

\begin{proof}
Suppose that $X$ is a countable structure which admits classification of open subgroups. 

\begin{claim}
	For every open subgroup $\Delta\subset\aut(X)$ there is a set $y\in\hf(X)$ such that $\Delta=\stab(y)$.
\end{claim}

\begin{proof}
	Use the classification assumption to find a finite set $a\subset X$ such that $\pstab(a)\subseteq\Delta\subseteq\stab(a)$. Fix a linear ordering $L$ of the set $a$ and let $y=\{\delta\cdot L\colon \delta\in\Delta\}$. This set consists of linear orders on the set $a$, therefore it is a finite set. It will be enough to show that $\Delta=\stab(y)$.
	
	The left-to-right inclusion is obvious from the definition of the set $y$. For the opposite inclusion, consider any element $\gamma\in\stab(y)$. Then there must be $\delta\in\Delta$ such that $\gamma\cdot L=\delta\cdot L$; in other words, $L=\delta^{-1}\gamma\cdot L$. As finite linear orders have no nontrivial automorphisms, this means that $\delta^{-1}\gamma\in\pstab(a)\subseteq\Delta$ holds. Then $\gamma=\delta(\delta^{-1}\gamma)\in\Delta$ holds as desired.
\end{proof}

\noindent In particular, $\aut(X)$ has only countably many open subgroups. Let $\langle\Delta_n\colon n\in\gw\rangle$ be a list of open subgroups of $\aut(X)$ which contains exactly one group from each conjugacy class.  For each $n\in\gw$, use the classification assumption to find $y_n\in\hf(X)$ such that $\Delta_n=\stab(y_n)$. Replacing $y_n$ with $\langle y_n, n\rangle$ if necessary, we may assume that the $\aut(X)$ orbits of $y_n$ are pairwise disjoint. 

Let $A\in W[[X]]$ be any set and work to find an ordinal $\ga$ and an injection from $A$ to $\hf(X)\times\ga$. Replacing $A$ with its superset $\{\gamma\cdot B\colon \gamma\in\aut(X), B\in A\}$ if necessary, assume that $\stab(A)=\aut(X)$. Observe that for every $B\in A$, $\stab(B)$ contains an open subgroup by the definition of $W[[X]]$, so it is open itself. In addition, for every $\gamma\in\aut(X)$, $\stab(\gamma\cdot B)=\gamma\cdot\stab(B)\gamma^{-1}$ holds, and so in each $\aut(X)$-orbit the stabilizers form exactly one conjugacy class of open subgroups of $\aut(X)$.

This makes it possible, in the surrounding universe $V[[X]]$ in which Axiom of Choice holds, to find a set $Z\subset A$ which selects from each $\aut(X)$-orbit exactly one element, and such that for each $B\in Z$ there is an $n\in\gw$ such that $\stab(B)=\Delta_n$. For each $n\in\gw$ write $Z_n=\{B\in Z\colon \stab(B)=\Delta_n\}$ and choose an ordinal $\ga_n$ and an injection $f\colon Z_n\to\ga_n$. Consider the set $f=\{\langle \gamma\cdot B, \langle \gamma\cdot y_n, f_n(B)\rangle\rangle\colon n\in\gw, B\in Z_n, \gamma\in\aut(X)\}$. It is clear that $f$ consists of elements of $W[[X]]$ and it is $\aut(X)$-invariant; therefore, it belongs to $W[[X]]$. It will be enough to prove that $f$ is a function and an injection; for then, it is an injection from $A$ to $\hf(X)\times\ga$, where $\ga=\sup_n\ga_n$.

To show that $f$ is a function and an injection, it is enough to show that for each set $B\in Z_n$ and group elements $\gamma, \delta\in\aut(X)$, $\gamma\cdot B=\delta\cdot B$ is equivalent to $\gamma\cdot y_n=\delta\cdot y_n$. Since $\stab(B)=\stab(y_n)=\Delta_n$, both of these equalities are equivalent to $\gamma^{-1}\delta\in\Delta_n$ and therefore are equivalent to each other. This concludes the proof.
\end{proof}

Now, we exploit the conclusion of Theorem~\ref{theoremD} abstractly to obtain conclusions about the theory of the resulting permutation model. For the first theorem, a \emph{tournament} on a set $Z$ is a directed graph on $Z$ which for any two distinct elements $z_0, z_1\in Z$ contains exactly one of the pairs $\langle z_0, z_1\rangle$ and $\langle z_1, z_0\rangle$. Clearly, existence of a tournament on every set is equivalent to the existence of a selector on any set of pairs. 

\begin{theorem}
	\label{abstracttheorem}
	\textnormal{(ZF or ZFA)} Suppose that $X$ is a set such that $\hf(X)$ is a support set. Then
	
	\begin{enumerate}
		\item if $X$ is linearly orderable then every set is linearly orderable \textnormal{\cite[Form 30]{howard:ac}};
		\item if there is a selector on $[X]^{<\aleph_0}$, then there is a selector on any family of nonempty finite sets \textnormal{\cite[Form 62]{howard:ac}};
		\item if there is a tournament on $X$ then there is a tournament on every set \textnormal{\cite[Form 88]{howard:ac}}.
	\end{enumerate}
\end{theorem}

\begin{proof}
	We need a propagation claim which is interesting in its own right.
	
	\begin{claim}
		\label{propagationclaim1}
		The following properties propagate from any set $Z$ to $\hf(Z)$:
		
		\begin{enumerate}
			\item $Z$ is linearly orderable;
			\item $[Z]^{<\aleph_0}$ has a selector;
			\item there is a tournament on $Z$.
		\end{enumerate}
	\end{claim}

	\begin{proof}
		For (1), suppose that $\leq$ is a linear ordering on $Z$. $[Z]^{<\aleph_0}$ can then be linearly ordered by $a\prec b$ if $a\subseteq b$ or the $\leq$-least element of $a\setminus b$ comes before the $\leq$-least element of $b\setminus a$ in the $\leq$-order. Iterating this construction, one can obtain a linear ordering on $\hf(Z)$.
		
		For (2), let $F$ be a selector on $[Z]^{<\aleph_0}$. To define a selector $G$ on the set of all nonempty finite subsets of $[Z]^{<\aleph_0}$, let $a$ be such a set. By recursion on $n\in\gw$ define elements $z_n\in Z$ and sets $a_n\subseteq a$ by the demands $a_0=a$, $z_n=F((\bigcup a_n)\setminus\{z_m\colon m\in n\})$ and $a_{n+1}=\{b\in a_n\colon z_n\in b\}$. The construction must terminate at some $n$ as $\bigcup a$ is finite. The only way how it can terminate is that $\bigcup a_n=\{z_m\colon m\in n\}$, and this set then must be the unique element of $a_n$. Define $G(z)$ to be this element. $G$ is clearly the desired selector on $[[Z]^{<\aleph_0}]^{<\aleph_0}$. Iterating this construction, one can obtain a selector on all finite subsets of $\hf(Z)$.
		
		(3) is more challenging. The key abstract rigidity observation is the following:
		
		\begin{enumerate}
			\item[(*)] whenever $C$ is a finite set and $T$ a tournament on it, there are no two distinct sets $A_0, A_1\subset C$ and an automorphism $\phi$ of $T$ such that $\phi(A_0)=A_1$ and $\phi(A_1)=A_0$.
		\end{enumerate}
		
		\noindent (*) is proved by contradiction. Suppose $C, T, A_0, A_1, \phi$ violate the conclusion. Let $c\in A_0\setminus A_1$ be any element. By the finiteness assumption, there must be a smallest number $n$ such that $\phi^n(x)=x$. This number must be even, since the values $\phi^m(x)$ must alternate between $A_0\setminus A_1$ and $A_1\setminus A_0$ as $m$ varies over natural numbers. Let $m$ be such that $n=2m$, and observe that $\phi$ flips the orientation of the oriented edge in $T$ connecting $c$ and $\phi^m(c)$. This contradicts the assumption that $\phi$ is an automorphism of $T$.
		
		Now, as in the previous cases, we will show that if there is a tournament on $Z$ then there is a tournament on $[Z]^{<\aleph_0}$ and then extend the construction to $\hf(Z)$. First, fix a list $T_n\colon n\in\gw$ of tournaments on finite sets such that each tournament on a finite set is isomorphic to exactly one tournament on the list. For each $n\in\gw$ select a tournament $S_n$ on $\power(\dom(T_n))$. This is all possible to do using a well-ordering on $V_\gw$. Now, if $T$ is a tournament on $Z$, define a tournament $S$ on $[Z]^{<\aleph_0}$ by putting $\langle a_0, a_1\rangle\in S$ if there is an isomorphism between $T\restriction a_0\cup a_1$ and some $T_n$ which moves $\langle a_0, a_1\rangle$ to some pair in $S_n$. The key observation (*) shows that this is a well-defined tournament. (3) follows.
	\end{proof}
	
	\noindent Now, suppose that $X$ is a set such that $Y=\hf(X)$ is a support set. For (1), suppose that $X$ is linearly orderable; by Claim~\ref{propagationclaim1} (1), there is a linear order $\leq$ on $Y$. Let $A$ be any set, and let $\ga$ be an ordinal and $f\colon A\to Y\times\ga$ be an injection.  Then $Y\times\ga$ can be linearly ordered by the lexicographic product of $\leq$ and the usual well-ordering on $\ga$, and the $f$-pullback of the ordering linearly orders $A$.
	
	For (2), suppose that there is a selector on the set $[X]^{<\aleph_0}$; by Claim~\ref{propagationclaim1} (2), there is a selector $F$ on the set $[Y]^{<\aleph_0}$. Let $A$ be any set consisting of finite sets, and let $\ga$ be an ordinal and $f\colon \bigcup A\to Y\times\ga$ be an injection.  For each set $a\in A$, consider $b_a=\{f(z)_0\colon z\in a\}$; this is a nonempty finite subset of $Y$. Let $G(a)=$such $z\in a$ such that $f(z)_0=F(b_a)$ and $f(z)_1$ is minimal possible. Clearly, $G$ is a selector on $A$.
	
	For (3), suppose that there is a tournament on $X$; by Claim~\ref{propagationclaim1} (3), there is a tournament $T$ on $Y$. 
	Let $A$ be any set consisting of finite sets, and let $\ga$ be an ordinal and $f\colon \bigcup A\to Y\times\ga$ be an injection. Define a tournament $S$ on $A$ by putting $\langle B, C\rangle\in S$ if either $f(B)_1\in f(C)_1$, or $f(B)_1=f(C)_1$ and $\langle f(B)_0, f(C)_0\rangle\in T$, for any two distinct elements $B, C\in A$. It is immediate that $S$ is a tournament on $A$.
\end{proof}

\begin{example}
	Let $X$ be the set of rational numbers with their usual ordering. This is the limit of the Fraiss{\'e} class of finite linear orders. This class is bootstrapping by Example~\ref{orderexample}, so by Theorem~\ref{bootstrappingtheorem} $X$ admits classification of open subgroups. By Theorem~\ref{theoremD}, in the permutation model $W[[X]]$, $\hf(X)$ is a support set. By Theorem~\ref{abstracttheorem}(1) applied inside $W[[X]]$, every set is linearly orderable in $W[[X]]$. This is the well-known Mostowski's linearly ordered model \cite[Model $\mathcal{N}3$]{howard:ac}. 
\end{example}

\begin{example}
	\label{lauchliexample}
	Consider the Fraiss{\'e} class of selectors on finite sets, and its limit $X$. The class is bootstrapping by Example~\ref{selectorexample}, so by Theorem~\ref{bootstrappingtheorem} $X$ admits classification of open subgroups. By Theorem~\ref{theoremD}, in the permutation model $W[[X]]$, $\hf(X)$ is a support set.  By Theorem~\ref{abstracttheorem}(2) applied inside $W[[X]]$, every set of nonempty finite sets has a selector in $W[[X]]$. This is the L{\" a}uchli model I \cite[Model $\mathcal{N}7$]{howard:ac} restated in rather more efficient terms.
\end{example}

\begin{example}
	Consider the Fraiss{\' e} class of tournaments on finite sets, and its limit $X$. It is not difficult to show that the class is bootstrapping.  By Theorem~\ref{bootstrappingtheorem} $X$ admits classification of open subgroups. By Theorem~\ref{theoremD}, in the permutation model $W[[X]]$, $\hf(X)$ is a support set.  By Theorem~\ref{abstracttheorem}(3) applied inside $W[[X]]$, there is a tournament on every set.
\end{example}

\noindent The point here is that the same arguments apply immediately to other models. Thus, in the Mathias--Pincus model II \cite[Model $\mathcal{N}5$]{howard:ac} in which there is a partial ordering and a linear ordering on the set of atoms, every set is linearly orderable basically by a succession of Examples~\ref{orderexample}, ~\ref{posetexample}, and~\ref{superpositionexample} and Theorem~\ref{abstracttheorem}(1).

All permutation models obtained from limits of bootstrapping classes share certain features which are not necessarily true in the general classification of open subgroups case. Here is the key notion.

\begin{definition}
	Let $X$ be a set.

\begin{enumerate}
\item A \emph{cone} on $X$ is a set $\{b\in [X]^{<\aleph_0}\colon a\subseteq b\}$ for some finite set $a\subseteq X$ referred as the base of the cone;
\item $X$ \emph{has no algebraicity} if every set $A\subseteq [X]^{<\aleph_0}$ closed under intersection either contains a cone or is disjoint from a cone.
\end{enumerate}
\end{definition}

\begin{proposition}
\label{algproposition}
	\textnormal{(ZF or ZFA)}
	Let $X$ be a set with no algebraicity.
	
	\begin{enumerate}
		\item every subset of $X$ has no algebraicity;
		\item every equivalence relation on $X$ with all classes finite has only finitely many non-singleton classes;
		\item there is no injection from $\gw$ to $[X]^{<\aleph_0}$;
\item for every infinite set $A\subseteq [X]^{<\aleph_0}$ there is a set $b\subset X$ such that the set $\{x\in X\colon b\cup\{x\}\in A\}$ is infinite.
	\end{enumerate}
\end{proposition}

\begin{proof}
For (1), suppose towards a contradiction that $Y\subseteq X$ is a set and $A\subseteq [Y]^{<\aleph_0}$ is a set closed under intersection which does not contain a cone and is not disjoint from a cone. The set $B\subseteq [X]^{<\aleph_0}$ given by $b\in B\liff b\cap Y\in A$ is closed under intersection and it does not contain a cone and it is not disjoint from a cone, violating the assumption on $X$.

 For (2), let $E$ be an equivalence relation on $X$ with all classes finite. Let $A\subseteq [X]^{<\aleph_0}$ be the set of all finite $E$-saturated sets. It is clearly closed under intersections. It cannot be disjoint from a cone because for any finite set $a\subset X$, the $E$-saturation of $a$ is a superset of $a$ which belongs to $A$. So, $A$ must contain the cone of all finite supersets of some finite set $a\subseteq X$. The only way that can occur is when all elements of $X\setminus a$ are pairwise inequivalent.

	For (3), suppose towards a contradiction that $g\colon \gw\to [X]^{<\aleph_0}$ is an injection. Manipulating the injection if necessary, we may assume that it is inclusion-increasing, and for each $n\in\gw$, $|g(n+1)\setminus g(n)|\geq 2$. Let $Y=\bigcup_ng_n$. The set $A\subseteq [X]^{<\aleph_0}$ defined by $b\in A\liff\exists n\ b\cap Y=g(n)$ is closed under intersections and it does not contain a cone and it is not disjoint from one, violating the assumption on $X$.

For (4), for each set $b\subset X$ write $b'=\{x\in X\colon b\cup\{x\}\in A\}$ if the latter set is finite, and $b'=0$ otherwise. Let $B=\{c\in [X]^{\aleph_0}\colon\forall b\subseteq c\ b'\subseteq c\}$. It is immediate that the set $B$ is closed under intersections. First note that $B$ cannot be disjoint from a cone. If $a\subset X$ was the base of the cone, one can define sets $a_n$ by recursion by $a_0=a$ and $a_{n+1}=a_n\cup\{b'\colon b\subseteq a_n\}$. By (3), this sequence of finite sets has to stabilize, and the stable value is then necessarily a superset of $a$ in $B$. Thus, $B$ contains a cone with a base $a$. Since the set $A$ is infinite, it must contain a set $c$ which is not a subset of $a$. Let $x\in c\setminus a$ be arbitrary and let $b=c\setminus\{x\}$. The set $a\cup b$ belongs to the set $B$, and this can occur only if $x\notin b'$; so, the set $\{x\in X\colon b\cup\{x\}\in A\}$ is infinite as desired.
\end{proof}

\begin{theorem}
	\label{algebraicitytheorem}
	\textnormal{(ZF or ZFA)} Suppose that $X$ is a set with no algebraicity such that $\hf(X)$ is a support set. Then
	
	\begin{enumerate}
		\item every set is either well-orderable or it contains an infinite set with no algebraicity, implying \textnormal{\cite[Form 63]{howard:ac}};
		\item a union of well-orderable collection of well-orderable sets is well-orderable \textnormal{\cite[Form 231]{howard:ac}};
		\item every infinite set of well-orderable sets has an infinite partial selector \textnormal{\cite[Form 380]{howard:ac}};
		\item on every collection of nonempty well-orderable sets there is a function assigning to each set one of its nonempty finite subsets \textnormal{\cite[Form 329]{howard:ac}}.
	\end{enumerate}
\end{theorem}

\begin{proof}
	Prior to further considerations, we must construct a function $\supp\colon \hf(X)\to [X]^{<\aleph_0}$ such that for every $y\in\hf(X)$, $y\in\hf(\supp(y))$. This is not difficult. First, by recursion on $n\in\gw$ define a function $f$ on $\hf(X)\times\gw$ by $f(y, 0)=\{y\}$ and $f(y, n+1)=f(y, n)\cup \bigcup\{z\in f(y, n)\colon z$ is finite$\}$. All values of $f$ are finite sets. The collection $Y=\{y\in\hf(X)\colon\exists n\ y\in\hf(f(y, n)\cap X)\}$ is easily seen to contain $X$ and to be closed under formation of finite subsets. So, it is enough to set $\supp(y)=f(y, n)\cap X$ for the least $n$ such that $y\in\hf(f(y, n)\cap X)$.
	
	For (1), let $A$ be any set. Let $\ga$ be an ordinal and $f\colon A\to \hf(X)\times\ga$ be an injection. Let $g\colon A\to [X]^{<\aleph_0}$ be the function defined by $g(B)=\supp(f(B)_0)$. Consider the set $D=\{g(B)\colon B\in A\}$. There are two cases. Either, the set $D$ is finite. In this case, select a linear ordering $L$ on $\bigcup D$ and naturally extend it to a well-ordering $L^*$ of $\hf(\bigcup D)$. Then the set $A$ is well-ordered by the $f$-pullback of the lexicographic product of the ordering $L^*$ with the usual ordering on $\ga$. Or, the set $D$ is infinite. In this case, use Proposition~\ref{algproposition}(4) to find a set $a\subset X$ such that the set $E=\{x\in X\colon a_x=a\cup\{x\}\in D\}$ is infinite. Select a linear ordering $L$ on $a$, and for each $x\in E$, extend it to $L_x$ on $a_x$ by appending $x$ to the end of $L$. The linear order $L_x$ naturally extends to a well-order $L^*_x$ of $\hf(a_x)$. Now, consider the injection $h\colon E\to A$ defined by $h(x)=$that element $B\in A$ for which $g(B)=a_x$ and $f(B)$ is the smallest possible in the lexicographic product of the ordering $L^*_x$ with the usual well-ordering on $\ga$. Clearly, the range of $E$ is an infinite subset of $A$ with no algebraicity.
	
	For (2), let $A$ be a well-orderable set consisting of well-orderable sets; without loss, assume that $A$ consists of pairwise disjoint sets. Suppose towards a contradiction that $\bigcup A$ is not well-orderable. By (1), there is an infinite set $C\subset \bigcup A$ with no algebraicity. Let $E$ be the equivalence relation on $C$ connecting elements which belong to the same set in $A$. This is an equivalence relation with well-orderable classes, so all but finitely many of them are singletons. Let $F=\{B\in A\colon B\cap C$ is a singleton$\}$. This is a subset of $A$ which is in a bijection with an infinite subset of $F$, so it has no algebraicity. Thus, $A$ cannot be well-orderable, a contradiction.
	
	For (3), let $A$ be a set consisting of well-orderable sets; without loss, assume that $A$ consists of pairwise disjoint sets. (If not, replace each set $B\in A$ with $\{B\}\times B$.) Either, the set $\bigcup A$ is well-orderable; in this case we get even a total selector on $A$. Or, the set $\bigcup A$ is not well-orderable; by (1), there is an infinite set $C\subset A$ with no algebraicity. Let $E$ be the equivalence relation on $C$ connecting elements which belong to the same set in $A$. This is an equivalence relation with well-orderable classes, so all but finitely many of them are singletons. Let $F=\{B\in A\colon B\cap C$ is a singleton$\}$. This is an infinite subset of $A$ and the set $C$ is a selector on it.
	
	For (4), let $A$ be a set consisting of nonempty well-orderable sets; without loss assume that $A$ consists of pairwise disjoint sets. Let $\ga$ be an ordinal and $f\colon \bigcup A\to\hf(X)\times\ga$ be an injection; for every $C\in B\in A$ write $g(C)=\supp(f(C)_0)$. For every $B\in A$, the set $g''B\subset [X]^{<\aleph_0}$ is an image of a well-orderable set, therefore well-orderable itself, and by the lack of algebraicity in $X$ finite. Now, every linear ordering $L$ on the finite set $a_B=\bigcup g''B$ canonically extends to a well-ordering on $\hf(a_B)$; observe that $\hf(a_B)$ contains $g''B$ as a subset. Thus, $L$ defines a single element $h(L, B)\in B$ as that $C\in B$ for which the value of $f(C)$ is smallest possible in the lexicographic product of the ordering $L^*$ and the usual ordering on $\ga$. There are only finitely many linear orderings on a finite set. Thus, the function $G$ on $A$ defined by $G(B)=\{h(L, B)\colon L$ is a linear order on $a_B\}$ assigns to every set in $B$ its nonempty finite subset. 
\end{proof}

\begin{corollary}
	\label{c1}
	If $\mathcal{F}$ is a bootstrapping Fraiss{\' e} class and $X$ is its limit, then in the associated permutation model $W[[X]]$, $X$ has no algebraicity, and the conclusions of Theorem~\ref{algebraicitytheorem} hold in it.
\end{corollary}

\begin{proof}
	Let $A\subseteq [X]^{<\aleph_0}$ be a set in $W[[X]]$ closed under intersections; we must prove that it is disjoint from or contains a cone. Let $a\subset X$ be a finite set such that $\pstab(a)\subseteq\stab(A)$. It will be enough to show that for all finite sets $a\subseteq b\subseteq c$, if $c\in A$ then $b\in A$. To see this, use the ultrahomogeneity features of $X$ to find an automorphism $\gamma\in\pstab(b)$ such that $c\cap \gamma\cdot c=b$. As the set $A$ is invariant under the action by $\pstab(b)$, it follows that $\gamma\cdot c\in A$. As $A$ is closed under intersections, $b=c\cap\gamma\cdot c\in A$ holds as desired.
\end{proof}

\begin{example}
	Let $X$ be a limit of the Fraiss{\' e} class of partial orderings. Then the conclusions of Theorem~\ref{algebraicitytheorem} all hold in the associated permutation model $W[[X]]$, which appears to be a new conclusion in each case. This is the Mathias--Pincus model I \cite[Model $\mathcal{N}3$]{howard:ac}.
\end{example}

\noindent An additional corollary exploits a very common amalgamation feature of Fraiss{\' e} classes.

\begin{definition}
	\cite{z:dideals}
	Let $\mathcal{F}$ be a Fraiss{\' e} class with disjoint amalgamation. Say that $\mathcal{F}$ has \emph{hereditary canonical amalgamation} if there is a function $C$ which to each ordered pair of $\mathcal{F}$-structures $\langle A, B\rangle$ in amalgamation position assigns a minimal disjoint amalgamation so that
	
	\begin{enumerate}
		\item $C$ is invariant under isomorphism in both variables: if $\phi\colon A\to A'$ and $\psi\colon B\to B'$ are isomorphisms which agree on $\dom(A)\cap \dom(B)$, then there is an isomorphism of $C(A, B)$ to $C(A', B')$ extending $\phi\cup\psi$;
		\item $C$ is hereditary for substructures in the left variable: if $A'$ is an induced substructure of $A$ such that $\dom(A)\cap\dom(B)\subseteq\dom(A')$, then there is an isomorphism between $C(A', B)$ and the algebraic closure of $\dom(A')\cup\dom(B)$ in $C(A, B)$ which is the identity on $\dom(A')\cup\dom(B)$.
	\end{enumerate}
\end{definition}

\noindent For example, the classes of linear orderings, partial orderings, metric spaces, or all the hereditary classes have hereditary canonical amalgamation, namely the amalgamation indicated in the corresponding examples in Section~\ref{bootstrappingsection}. The collection of relational Fraiss{\'e} classes with canonical hereditary amalgamation is closed under superposition.

\begin{corollary}
	If a Fraiss{\' e} class $\mathcal{F}$ is bootstrapping and has canonical hereditary amalgamation and $X$ is its limit, then in the associated permutation model $W[[X]]$, every set carries a nonprincipal ultrafilter \textnormal{\cite[Form 63]{howard:ac}}.
\end{corollary}
	
\begin{proof}
	In view of Corollary~\ref{c1}, $X$ is a set with no algebraicity and $\hf(X)$ is a support set in the model $W[[X]]$. The proof of Theorem~\ref{algebraicitytheorem}(1) shows that every set in $W[[X]]$ is either well-orderable or it contains an injective image of an infinite subset of $X$. Thus, it will be enough to show that every infinite subset of $X$ carries an ultrafilter.
	
	Let $A\subset X$ be an infinite set, let $a\subset X$ be a finite set such that $\pstab(a)\subseteq\stab(A)$. Let $x_0\in A\setminus a$ be any point; shrinking the set $A$ if necessary, we may assume that $A$ is just the $\pstab(a)$-orbit of $x_0$. For every finite set $b\subset X$ let $B_b=\{x\in A\setminus b\colon b$ is canonically amalgamated with $\{x\}$ over $a\}$; let $U$ be the filter generated by the sets $B_b$ as $b$ varies over all finite subsets of $X$. It will be enough to show that $U$ is a nonprincipal ultrafilter.
	
	To do this, first note that the sets $B_b$ form a centered system. If $\{b_i\colon i\in n\}$ is a finite collection of finite subsets of $X$, then $B_{\bigcup_ib_i}\subset \bigcap_iB_{b_i}$ by the heredity clause in the canonical amalgamation definition. Second, suppose that $C\subset X$ is any set; we must find a finite set $b\subset X$ such that $B_b\subseteq C$ or $B_b\cap C=0$. To do this, just find a finite set $b$ such that $a\subseteq b$ and $\pstab(b)\subseteq\stab(C)$. There must be a point $x\in A\setminus b$ which is amalgamated to $b$ over $a$ using the canonical amalgamation; the set $B_b$ of all such points is $\pstab(b)$-invariant. As such, it is a subset of either $C$ or its complement, since $C$ is also $\pstab(b)$-invariant. This completes the proof.
\end{proof}

\begin{example}
	Let $X$ be a limit of the Fraiss{\' e} class of partial orderings superposed with a linear ordering. In the associated model $W[[X]]$, every infinite set carries a nonprincipal ultrafilter, an apparently new conclusion. This is the Mathias--Pincus model II \cite[Model $\mathcal{N}5$]{howard:ac}.
\end{example}

\noindent For the last theorem, we isolate a neat criterion ensuring that in the permutation model there are as few linearly orderable sets as possible. 
This has been studied recently in \cite{tachtsis:order} using quite different methods.

\begin{definition}
	A set $X$ is \emph{L-finite} if for every $n\in\gw$, every linear pre-order on $[X]^{\leq n}$ has only finitely many classes.
\end{definition}

\begin{theorem}
	\label{lineartheorem}
	Suppose that $X$ is a set such that $\hf(X)$ is a support set. If $X$ is L-finite then every linearly orderable set is well-orderable \textnormal{\cite[Form 90]{howard:ac}}.
\end{theorem}

\begin{proof}
We need a propagation claim which interesting in its own right.

\begin{claim}
	\label{propagationclaim2}
	Suppose that $Z$ is an L-finite set. Then for every $n\in\gw$, the set $[Z]^{\leq n}$ is L-finite as well. 
\end{claim}

\begin{proof}
	Suppose that $m\in\gw$ is a number and $\leq$ be a linear preorder on $[[Z]^{\leq m}]^{\leq n}$. Write $E$ for the $\leq$-equivalence relation, $Q$ for the quotient $E$-space, and use the symbol $\leq$ for the quotient linear ordering on $Q$. We need to show that $Q$ is finite.
	
	There is a number $k\in\gw$ such that for every set $a\in [Z]^{nm}$, the set $U_a=[[a]^{\leq m}]^{\leq n}$ has cardinality $k$. Let $\langle q_{ia}\colon i\in k_a\rangle$ list all $E$-equivalence classes of sets in $U_a$ in $\leq$-decreasing order; note that $k_a\leq k$ holds. Let $f_i\colon [Z]^{nm}\to Q$ be the partial map sending $a$ to $q_{ia}$, for every $i\in k$. For each $i\in k$, let $\leq_i$ be the preorder on $[Z]^{nm}$ defined by $a\leq_i b$ if $f_i(a)\leq f_i(b)$ or $f_i(a)$ is undefined. $\leq_i$ is a linear preorder on $[Z]^{nm}$.
	
	By the assumption on the set $Z$, the preorder $\leq_i$ on $[Z]^{nm}$ has only finitely many classes, meaning that the range of $f_i$ is a finite subset of $Q$. It will be enough to show that $Q=\bigcup_i\rng(f_i)$.  For this, let $u\in [[Z]^{\leq m}]^{\leq n}$ be any set, and let $a\subset Z$ be any set of cardinality $nm$ which is a superset of $\bigcup u$. It is clear that $u\in U_a$ holds, so for some $i\in k_a$ the $E$-class of $u$ must be listed as $q_{ia}$, and must be in the range of $f_i$.
\end{proof}

\noindent Now suppose that $X$ is L-finite and $\hf(X)$ is a support set. Let $A$ be a set with a linear ordering $\leq$. To show that $A$ is well-orderable, it will be enough to present $A$ as the union of a well-orderable collection of finite sets, as these finite sets are all linearly ordered by $\leq$. Let $\ga$ be an ordinal and $f\colon A\to Y\times\ga$ be an injection. For each ordinal $\gb\in\ga$ write $g_\gb\colon Y\to A$ for the partial function such that $f(g_\gb(y))=\langle y, \gb\rangle$. By recursion on $k\in\gw$ define sets $Y_k$ by $Y_0=0$ and $Y_{k+1}=Y_k\cup [Y_k]^{\leq k}$. It is clear that $Y=\bigcup_kY_k$. For each number $k$, the set $Y_k\cap\dom(g_\gb)$ must be finite, because it is linearly ordered by the $g_\gb$-pullback of $\leq$, and the assumption on $X$ and Claim~\ref{propagationclaim2} show that every linear preorder on $Y_k$ has only finitely many classes. Thus, we obtain $A=\bigcup_{\gb\in\ga}\bigcup_{k\in\gw} g_\gb''Y_k$ as desired.
\end{proof}

 \noindent How common are structures in which the assumptions of Theorem~\ref{lineartheorem} are satisfied? We close the paper by isolating a large group of such examples.

\begin{definition}
	Let $\mathcal{F}$ be a Fraiss{\' e} class in relational language and disjoint amalgamation.
	
	\begin{enumerate}
		\item Let $A\subseteq B$ be structures in $\mathcal{F}$. A \emph{cyclic amalgamation of $B$ over $A$} is a disjoint minimal amalgamation $C$ of two copies of $B$ which intersect in $A$ such that whenever $B_0, B_1, B_2$ are copies of $B$ which pairwise intersect in $A$ and $D$ is a minimal disjoint amalgamation of $B_0, B_1$, then there is a minimal disjoint amalgamation $E$ of $B_0, B_1, B_2$ such that $E\restriction B_0, B_1=D$, and both $E\restriction B_1, B_2$ and $E\restriction B_2, B_0$ are isomorphic to $C$;
		\item $\mathcal{F}$ is \emph{cyclic} if for any pair $A\subseteq B$ of structures in $\mathcal{F}$ there is a cyclic amalgamation of $B$ over $A$.
	\end{enumerate}
\end{definition}

\noindent For the statement of the central theorem, recall that a Fraiss{\' e} class in (possibly infinite) relational language is \emph{locally finite} if for every $n$ it contains only finitely many structures of cardinality $n$ up to isomorphism.

\begin{theorem}
	\label{cyclictheorem}
	Let $\mathcal{F}$ be a cyclic locally finite Fraiss{\' e} class. Let $X$ be its Fraiss{\' e} limit. Then in the permutation model $W[[X]]$, for every number $n\in\gw$, every linear pre-order on $[X]^{\leq n}$ has finitely many classes.
\end{theorem}

\begin{proof}
Let $n\in\gw$ be a number and $\leq$ be a linear preorder on $[X]^{\leq n}$ in the model $W[[X]]$. Write $\equiv$ for the equivalence relation induced by $\leq$. Let $A\subset X$ be a finite set such that $\pstab(A)\subset\stab(\leq)$. It will be enough to show that for any set $B\in [X]^{\leq n}$, the $\equiv$-class of $B$ depends only on the isomorphism type of $A\cup B$, because by the local finiteness assumptions there are only finitely many of such isomorphism types available. 

Let $C$ be a cyclic amalgamation of $A\cup B$ over $A$. Below, the expressions of the form $B_i$ always stand for sets in $[X]^{\leq n}$ such that $A\cup B_i$ is isomorphic to $A\cup B$ by an isomorphism which fixes $A$. To start, observe that the invariance assumption on $\leq$ and ultrahomogeneity properties of the limit $X$ show that the status of $B_0\leq B_1$ depends only on the isomorphism type of $A\cup B_0\cup B_1$ over $A$. The following series of claims then proves the theorem.

 \begin{claim}
 	If $B_0\cap B_1\subseteq A$ and the structure $X\restriction A\cup B_0\cup B_1$ is isomorphic to $C$, then $B_0\equiv B_1$ holds.
 \end{claim}
 
 \begin{proof}
 	By the ultrahomogeneity properties of $X$ and the cyclic assumption on $C$ there is another copy $B_2$ such that the sets $B_1\cap B_2$ and $B_0\cap B_2$ are both subsets of $A$ and the structures $X\restriction A\cup B_1\cup B_2$ and $X\restriction A\cup B_2\cup B_0$ are both isomorphic to $C$.
 	
 	Now, by the presumed linearity of $\leq$, it must occur that either $B_0\leq B_1\leq B_2\leq B_0$ or $B_0\geq B_1\geq B_2\geq B_0$. Both cases mean the same thing: all the $B_i$ are $\equiv$-related, concluding the proof.
 \end{proof}
 
 \begin{claim}
 	If $B_0\cap B_1\subseteq A$, then $B_0\equiv B_1$ holds.
 \end{claim}
 
 \begin{proof}
 	By the ultrahomogeneity properties of $X$ and the cyclic assumption on $C$ there is another copy $B_2$ such that the sets $B_1\cap B_2$ and $B_0\cap B_2$ are both subsets of $A$ and the structures $X\restriction A\cup B_1\cup B_2$ and $X\restriction A\cup B_2\cup B_0$ are both isomorphic to $C$. By the previous claim $B_1\equiv B_2\equiv B_0$ holds, and the transitivity of $\equiv$ completes the proof.
 \end{proof}

\begin{claim}
	For any $B_0, B_1$, $B_0\equiv B$ holds.
\end{claim}

\begin{proof}
	By the ultrahomogeneity properties of $X$, there must be $B_2$ and $B_3$ such that the sets $B_0\cap B_2$, $B_1\cap B_3$, and $B_2\cap B_3$ are all subsets of $A$. By the previous claim, $B_0\equiv B_2\equiv B_3\equiv B_1$ holds. The transitivity of $\equiv$ completes the proof.
\end{proof}

\noindent The last claim shows that all isomorphic copies of $B$ occupy the same $\equiv$-class. This proves the theorem.
\end{proof}

\noindent Together with Theorems~\ref{theoremC}, ~\ref{theoremD}, and ~\ref{abstracttheorem}(6), we now get an attractive corollary.

\begin{corollary}
	If $\mathcal{F}$ is a cyclic, bootstrapping Fraiss{\' e} class and $X$ is its limit, then in the permutation model $W[[X]]$, every set is linearly orderable iff it is wellorderable, \textnormal{\cite[Form 90]{howard:ac}}.
\end{corollary}

\begin{example}
	Every hereditary Fraiss{\' e} class $\mathcal{F}$ with disjoint amalgamation is cyclic; in fact, the inclusion-smallest amalgamation as described in Example~\ref{hereditaryexample} is cyclic. To see that, let $A\subseteq B$ be structures in $\mathcal{F}$, $B_0, B_1, B_2$ be copies of $B$ pairwise intersecting in $A$, and let $D$ be a prescribed minimal disjoint amalgamation of $B_0, B_1$. By the disjoint amalgamation assumption, there is a minimal disjoint amalgamation $E$ of $B_0, B_1, B_2$ such that $E\restriction B_0\cup B_1=D$. Now remove from the relations in $E$ all subsets of $B_1\cup B_0$ and $B_2\cup B_1$ except for those which are subsets of $B_0, B_1$, or $B_2$. The resulting structure belongs to $\mathcal{F}$ by the heredity assumption, and it is easy to see that it amalgamates $B_0, B_1$ with $D$, and $B_1, B_2$ as well as $B_2, B_0$ in the smallest amalgamation.
	
	Thus, for example the permutation model derived from the Rado graph satisfies that a set is linearly orderable iff it is well-orderable as per Theorems~\ref{cyclictheorem}, ~\ref{abstracttheorem}(4), and~\ref{theoremC}.
\end{example}

\begin{example}
	The class of partial orders is cyclic; in fact, the inclusion-smallest amalgamation as described in Example~\ref{posetexample} is cyclic. The elementary verification is left to the reader. Thus, if $X$ is the Fraiss{\' e} limit of this class and $W[[X]]$ is the associated permutation model, then in $W[[X]]$, every set is linearly orderable iff it is well-orderable as per Theorems~\ref{cyclictheorem}, ~\ref{abstracttheorem}(4), and~\ref{theoremC}. This is the Mathias--Pincus model I \cite[Model $\mathcal{N}4$]{howard:ac}.
\end{example}

\begin{example}
	The class of selectors is cyclic; in fact, any disjoint amalgamation is cyclic. The elementary verification is left to the reader. Let $X$ be its limit; the associated permutation model is the L{\" a}uchli model I \cite[Model $\mathcal{N}7$]{howard:ac} discussed in Example~\ref{lauchliexample} above. In it, a set is linearly orderable iff it is well-orderable, an apparently new conclusion. 
\end{example}

\noindent The collection of cyclic Fraiss{\' e} classes is rather obviously closed under superposition. The class of finite linear orders is not cyclic, as one-element poset has no cyclic amalgamation over the empty set. After all, if $X$ is the limit ordering and $W[[X]]$ its associated permutation model, then in $W[[X]]$ the set $X$ is linearly orderable but not well-orderable, violating the conclusion of Theorem~\ref{abstracttheorem}(5).

\bibliographystyle{plain} 
\bibliography{odkazy,zapletal,shelah}

\end{document}